\documentclass{article}
\usepackage{times}
\usepackage{amsfonts}
\usepackage{amsmath}
\usepackage{amssymb}
\usepackage{tikz}
\usetikzlibrary{positioning}
\usetikzlibrary{arrows,decorations.markings}
\usetikzlibrary{shapes.geometric, arrows}
\usepackage[font=small,labelfont=bf]{caption}
\usetikzlibrary{intersections, calc, angles}
\usepackage{subcaption}
\usetikzlibrary{decorations.pathmorphing}
\usetikzlibrary{decorations.pathreplacing,calligraphy}
\usepackage{hyperref}
\usepackage{relsize}
\newtheorem{theorem}{Theorem}
\newtheorem{example}{Example}
 \newtheorem{corollary}{Corollary}
\newtheorem{definition}{Definition}
\newtheorem{remark}{Remark}
\newtheorem{proposition}{Proposition}
\newtheorem{lemma}{Lemma}
\newenvironment{proof}[1][Proof]{\noindent\textbf{#1.} }{\ \rule{0.5em}{0.5em}}
\tikzset{
bicolor/.style 2 args={
  thick,dashed,dash pattern=on 3pt off 3pt,#1,
  postaction={draw,dashed,thick,dash pattern=on 3pt off 3pt,#2,dash phase=3pt}
  },
swiggle/.style={
    -stealth,decorate,decoration={snake,amplitude=3pt,pre length=2pt,post length=3pt}
  }
}
\usepackage{fancyhdr}
\begin{document}

\title{Continuity of the roots of a nonmonic polynomial and applications in multivariate stability theory}
\markright{Root continuity of nonmonic polynomials}
\author{Anthony Stefan and Aaron Welters}
\author{Anthony Stefan\footnote{Email: astefan2015@my.fit.edu}\; and Aaron Welters\footnote{Email: awelters@fit.edu}\\Florida Institute of Technology\\Department of Mathematical Sciences\\Melbourne, FL, USA}
\date{\today}

\maketitle
\sloppy

\begin{abstract}
We study continuity of the roots of nonmonic polynomials as a function of their coefficients using only the most elementary results from an introductory course in real analysis and the theory of single variable polynomials. Our approach gives both qualitative and quantitative results in the case that the degree of the unperturbed polynomial can change under a perturbation of its coefficients, a case that naturally occurs, for instance, in stability theory of polynomials, singular perturbation theory, or in the perturbation theory for generalized eigenvalue problems. An application of our results in multivariate stability theory is provided which is important in, for example, the study of hyperbolic polynomials or realizability and synthesis problems in passive electrical network theory, and will be of general interest to mathematicians as well as physicists and engineers.
\end{abstract}

 The continuity of the roots of a \emph{monic} polynomial on its coefficients (belonging to the complex numbers $\mathbb{C}$) is usually proved in three different ways. The first is an elementary real analysis approach \cite{77US} using Bolzano-Weierstrass Theorem \cite[Theorem 2.42]{76WR}. The second approach \cite[Theorem 1.4]{66MM} uses Rouch\'{e}'s Theorem from complex analysis which is based on contour integration and the argument principle \cite[Theorem 1.3]{66MM}. The third uses a topology approach \cite{89CC, 20KH} which is based on the construction of a homeomorphism between the coefficients of the monic polynomial and its roots.
 
 Unfortunately, the simplest and most elementary approach from \cite{77US} does not apply for \emph{nonmonic} polynomials, which occurs in many important areas such as in singular perturbation theory \cite{99BO}, perturbation theory for generalized eigenvalue problems \cite{90SS}, and in the stability theory of polynomials and its applications \cite{90SB, 87FBII, 09BBI, 09BBII, 77NB, 17NB, 07PB, 04CO, 86AF, 84FL, 87FBI, 68TK, 91AK, 11RP, 11DW}. Furthermore, the qualitative and quantitative results for nonmonic polynomials have some distinctions (essentially due to the possibility of a difference between the degree of the perturbed polynomial and the degree of unperturbed polynomial) and certain modifications are needed. Moreover, although previous results \cite[Chap. 1, Sec. 1, Exercise 13]{66MM} and \cite[Theorem 3]{89CC} do have statements comparable to ours (cf.\ Theorem \ref{ThmRootSepincludingInfty}), they are proved using either Rouch\'{e}'s theorem \cite[Theorem 1]{65MZ} or by the topology approach \cite[Theorem 2]{89CC} and without any comparable quantitative result like ours (see Theorem \ref{ThmQuantOstrowskiLike}).
 
 This article aims to do the following:
\begin{itemize}
    \item Give an elementary approach to proving continuity of the roots of a nonmonic polynomial on its coefficients that an undergraduate student whose taken an introductory real analysis course could do the proof. 
    \item Provide both qualitative and quantitative results, in an elementary and self-contained manner, which can be used to introduce the reader to modern research in stability theory of multivariate polynomials and its applications.
    \item Include examples and visual illustrations of the analysis so that it is even more accessible to the reader.
\end{itemize} 

The rest of the article is organized as follows. In Section \ref{Qualitative Section} on the qualitative results, we introduce our notation for the unperturbed and perturbed polynomials (denoted by $q$ and $p$, respectively) and provide the elementary properties of single variable polynomials that are needed. Then we state and prove the main theorem of this article, namely, Theorem \ref{Thm1} along with two immediate corollaries (Corollaries \ref{CorPHasExactlymRoots} and \ref{cor2}). As our polynomials can be nonmonic, there is the possibility of some of the roots of the perturbed polynomial converging to $\infty$ and the main result in this regard is Theorem \ref{ThmRootSepincludingInfty}. We prove this after introducing the notation of a reciprocal polynomial to a given polynomial (Definition \ref{DefReciprocalPoly}) and proving some preliminary results using this (Proposition \ref{Prop1AtLeastZeroRoots} and Corollary \ref{Cor3HasExactlymRoots}). All these results are illustrated in Figures \ref{Fig1} and \ref{Fig2} with a summary given in Remark \ref{RemRootGroupings} which is illustrated in Figure \ref{FigRemarkConclusion}. 

Next, in Section \ref{SecQuantitativeResults} on our quantitative results, we begin with the main theorem for that section, namely, Theorem \ref{ThmQuantOstrowskiLike}. We immediately follow this up with an example (Example \ref{ExNumericalOfQuantitativeMainThm}) to help the reader understand the theorem. At the end of that section we provide some preliminary results (Lemma \ref{LemRootUpperBoundMonicPoly} and Proposition \ref{PropPrimPertBddMonicPolys}) needed to prove this theorem. 

In our final section, Section \ref{SecApplications}, we discuss the applications of our results in stability theory for multivariate polynomials that begins with notation and basic properties of multivariable polynomials and a concrete example (Example \ref{ExMultvarPolyNotation}) to help the reader understand this notation. Next, we give our main theorem (Theorem \ref{11RPThm}) of this section from \cite{11RP}, its proof using only the results presented in our paper, and conclude with a special case of the proof for two-variables (see Example \ref{Ex4m=2}) whose details are illustrated in Figures \ref{firstfigz1z2plane} and \ref{firstfigz1z2plane2}.

\section{Qualitative Results}\label{Qualitative Section}
Let $P_n(\mathbb{C})$ denote the set of all polynomials with coefficients in the complex numbers $\mathbb{C}$ whose degree is less than or equal to a positive integer $n$ (i.e., $n\in\mathbb{N}$) and include the zero polynomial. In particular, any two polynomials $p(z)$, $q(z)\in P_n(\mathbb{C})$ can be expressed in the following form:
\begin{align}
    p(z)&= a_nz^n + a_{n-1}z^{n-1} + \cdots + a_1z+a_0, \label{Def_Poly_p}\\
    q(z)&= b_nz^n + b_{n-1}z^{n-1} + \cdots + b_1z+b_0, \label{Def_Poly_q}
\end{align}
for some $a_i, b_i\in \mathbb{C}$ and $i=0,\ldots, n$. The degree of a nonzero polynomial $p$, denoted by $\deg p$, is defined in terms of its coefficients (\ref{Def_Poly_p}) by
\begin{align}
    \deg p=\max\{i:a_i\not = 0\}.
\end{align}

To prove our main theorem below, we will use the following basic results in analysis and the theory of polynomials in one variable:
\begin{enumerate}
    \item If $p(z)$ has $\deg p\geq 1$ then it has exactly $\deg p$ roots $\lambda_1, \ldots, \lambda_{\deg p}$ (counting multiplicities) and
    \begin{align}
        p(z)&=a_{\deg p}\prod_{i=1}^{\deg p}(z-\lambda_i), \label{ProdDegreeFactoriization}\\
        \frac{p'(z)}{p(z)}&=\sum_{i=1}^{\deg p}\frac{1}{z-\lambda_i}=\frac{na_nz^{n-1}+\cdots +a_1}{a_nz^n + a_{n-1}z^{n-1} + \cdots + a_1z+a_0}\;. \label{SumDegreeFactoriization}
    \end{align}
    \item If $a_i\to b_i$ for all $i=0,\ldots, n$ then for each fixed $z\in\mathbb{C}$, 
    \begin{align}
        &p(z)\to q(z),\;\;p'(z)\to q'(z), \text{ and } \notag\\
        &\frac{p'(z)}{p(z)}\to  \frac{q'(z)}{q(z)}, \text{ provided } q(z)\not = 0.\label{convervingfrompp'toqq'}
    \end{align}
    \item Every bounded sequence in $\mathbb{C}^N$ has a convergent subsequence (for any fixed $N\in\mathbb{N}$).
\end{enumerate}

\begin{theorem}[Main Theorem]\label{Thm1}
Suppose $\zeta$ is a root of $q(z)$ of multiplicity $m$ and let $\varepsilon >0$ be given. Then there exists a $\delta >0$ such that if $\max_{0\leq i\leq n}|a_i-b_i|<\delta$ then $p(z)$ has at least $m$ roots whose distance from $\zeta$ is less than $\varepsilon$.
\end{theorem}

\begin{proof}
Suppose not. Then, for each $k\in \mathbb{N}$ there exists a polynomial $p_k= a_n^{(k)}z^n + a_{n-1}^{(k)}z^{n-1} + \cdots + a_{1}^{(k)}z + a_0^{(k)}$ such that $p_k$ has at most $m-1$ roots whose distance from $\zeta$ is less then $\varepsilon$ and
\begin{align}
    \lim_{k\to \infty}a_i^{(k)}= b_i\qquad (0\leq i\leq n).
\end{align}
In particular, by hypotheses ($q$ has at least one root, namely, $\zeta$, with finite multiplicity $m$) we have $\deg q\geq 1$ and hence 
\begin{align}
    a_{\deg q}^{(k)} \to b_{\deg q}\not = 0 \text{ as } k \to \infty,
\end{align}
implying there exists a $K\in\mathbb{N}$ such that
\begin{align}
    \deg p_k\geq \deg q\geq 1
\end{align}
for all $k\geq K$. Also, since $\deg p_k\in \{\deg q,\ldots, n\}$, then there exists a $N\in \{\deg q,\ldots, n\}$ such that $\deg p_k=N$ for infinitely many $k$. Hence, by choosing a subsequence if necessary, we can assume that we also have $\deg p_k=N$ for every $k\geq K$. Now, for each integer $k\geq K$, let $\lambda_i^{(k)}$, $i=1,\ldots, N$ denote all the roots of $p_k$ (counting multiplicities). Then, statement 3 above implies the sequence $\{(\lambda_1^{(k)}, \ldots , \lambda_N^{(k)})\}$ has a subsequence $\{(\lambda_1^{(k_j)}, \ldots , \lambda_N^{(k_j)})\}$ converging to a point $(\mu_1,\ldots, \mu_N)$ in $(\mathbb{C} \cup \{\infty\})^N$. In particular, it follows that the number of $i$'s such that $\mu_i=\zeta$ must be at most $m-1$. This together with (\ref{SumDegreeFactoriization}) and (\ref{convervingfrompp'toqq'}) implies that
\begin{align}
    1=\lim_{z\to \zeta}\lim_{j\to \infty} \frac{(z-\zeta)\frac{p_{k_j}'(z)}{p_{k_j}(z)}}{(z-\zeta) \frac{q'(z)}{q(z)}}=\frac{\# \text{ of } i\text{'s such that }\mu_i=\zeta}{m}\leq \frac{m-1}{m},
\end{align}
a contradiction. This proves the theorem.
\end{proof}

\begin{figure}[!]
\begin{tabular}{cc}
    \begin{tikzpicture}[scale=1]
        \draw (0,-3) node {$(i=0,\ldots,n)$};
        \draw[dashed] [color=black] (0,0) circle (1.5);
        \node[right] (0,0) {$b_i$};
        \path [cyan,bend left, dashed]   (-1.11803398875,0) edge (-1,0.5);
        \draw [color=black, fill=black] (0,0)circle(0.05);
        \draw[cyan] (0,0) -- (-1,0.5) node[above] {\textcolor{black}{$a_i$}};
        \draw [pen colour={cyan}, decorate,
        decoration = {calligraphic brace, amplitude=6pt}] (-1,0.5)-- (0,0);
        \draw [color=black, fill=black] (-1,0.5)circle(0.05);
        \draw (-0.75,0) node[below] {$\delta$};
        \draw (0,0) -- (-1.5,0);
        \draw (-0.4,0.6) node[right] {\textcolor{cyan}{$|a_i-b_i|$}};
        \draw (0,2.5) node {$B(b_i,\delta)$};
        \draw (-0.3,-4) node {\textbf{(a)}};
    \end{tikzpicture}  
& 
        \begin{tikzpicture}[scale=0.9]
            \draw[dashed] (6,0)circle[radius=2];
            \draw (6,0) node[left] {$\zeta$} -- (8,0);
            \draw [color=black, fill=black] (6,0)circle(0.05);
            \draw (7,0) node[above] {$\varepsilon$};
            \draw [color=black, fill=black] (5.7,-1)circle(0.05) node[right]{$\lambda_{m}$};
            \draw [color=black, fill=black] (5.15,0.5)circle(0.05) node[left]{$\lambda_{1}$};
            \draw [color=black, fill=black] (5.35,-0.8)circle(0.02);
            \draw [color=black, fill=black] (5.08,-0.38)circle(0.02);
            \draw [color=black, fill=black] (5.08,0.1)circle(0.02);
            \draw[swiggle] { (8.37 ,-1.2 ) ++ (0 ,0 ) } -- (7.8,-1.9);
            \draw (8.37,-1.2) node[yshift=0.2cm, xshift=0.1cm] {$\lambda_{k}$};
            \draw (7.8,-1.9) node[below, yshift=-0.1cm, xshift=-0.1cm] {?};
            \draw[swiggle] {(8.26 ,-1 ) ++ (0 ,0 )} -- (7.3,-0.7); 
            \draw (7.3,-0.7) node[left, xshift=-0.1cm] {?};
            \draw [color=black, fill=black] (7.76,-1.97)circle(0.05);
            \draw [color=black, fill=black] (7.2,-0.68)circle(0.05);
            \draw (6,-3) node {$(m< k \leq \deg p)$};
            \draw (6,2.75) node {$B(\zeta,\varepsilon)$};
            \draw (6,-4) node {\textbf{(b)}};
        \end{tikzpicture} \vspace{0.25cm}
\\
    \begin{tikzpicture}[scale=0.9]
        \draw[dashed] (6,0)circle[radius=2];
        \draw (6,0) node[left] {$\zeta$} -- (8,0);
        \draw [color=black, fill=black] (6,0)circle(0.05);
        \draw (7,0) node[above] {$\varepsilon$};
        \draw [color=black, fill=black] (5.7,-1)circle(0.05) node[right]{$\lambda_{m}$};
        \draw [color=black, fill=black] (5.15,0.5)circle(0.05) node[left]{$\lambda_{1}$};
        \draw [color=black, fill=black] (5.35,-0.8)circle(0.02);
        \draw [color=black, fill=black] (5.08,-0.38)circle(0.02);
        \draw [color=black, fill=black] (5.08,0.1)circle(0.02);
        \draw[swiggle] {(8.37 ,-1.2 ) ++ (0 ,0 )} node[yshift=0.2cm, xshift=0.1cm] {$\lambda_{k}$} -- (7.8,-1.9);
        \draw (6,-4.6) node {$(m< k \leq \deg p)$};
        \draw [color=black, fill=black] (7.76,-1.97)circle(0.05);
       \draw (6,2.5) node {$B(\zeta,\varepsilon)$};
        \draw (6,-5.5) node {\textbf{(c)}};
    \end{tikzpicture} 
  & \hspace{0.5cm}
        \begin{tikzpicture}[scale=0.9]
            \draw (0,2.5) node {$B(\zeta_j,\varepsilon)$};
            \draw[dashed] [color=black] (0,0) circle (2);
            \node[left] (0,0) {$\zeta_j$};
            \draw [color=black, fill=black] (0,0)circle(0.05);
            \draw (1,0) node[above] {$\varepsilon$};
            \draw (0,0) -- (2,0);
            \draw [color=black, fill=black] (-0.85,-1.2)circle(0.05) node[right]{$\lambda_{j,m_j}$};
            \draw [color=black, fill=black] (-1.25,0.75)circle(0.05) node[right]{$\lambda_{j,1}$};
            \draw [color=black, fill=black] (-1.45,0.35)circle(0.02);
            \draw [color=black, fill=black] (-1.5,-0.2)circle(0.02);
            \draw [color=black, fill=black] (-1.25,-0.75)circle(0.02);
            \draw[dashed] (3.5,-2.25)circle[radius=2];
            \draw (3.5,0.5) node {$B(\zeta_i,\varepsilon)$};
            \draw (1.35,-2.5) node[left] {\textcolor{cyan}{$|\zeta_i-\zeta_j|$}};
            \draw (3.5,-2.25) node[below] {$\zeta_i$} -- (5.55,-2.25);
            \draw[cyan] (3.5,-2.25) -- (0,0);
            \path [cyan,bend right, dashed]   (3.5,-2.25) edge (4.19,0);
            \draw[cyan] (2,0) -- (4.19,0);
            \draw [color=black, fill=black] (3.5,-2.25)circle(0.05);
            \draw (4.55,-2.25) node[above] {$\varepsilon$};
            \draw [color=black, fill=black] (3.8,-3.525)circle(0.05) node[right]{$\lambda_{i,m_i}$};
            \draw [color=black, fill=black] (2.05,-2.5)circle(0.05) node[right]{$\lambda_{i,1}$};
            \draw [color=black, fill=black] (2.35,-2.9)circle(0.02);
            \draw [color=black, fill=black] (2.7,-3.3)circle(0.02);
            \draw [color=black, fill=black] (3.25,-3.45)circle(0.02);
            \draw [pen colour={cyan}, decorate, decoration = {calligraphic brace, amplitude=20pt}] (3.5,-2.25) -- (0,0);
            \draw[color=cyan] (1.375,-1.72) -- (1.23,-1.95);
            \draw (1,-4.6) node {$(1\leq i,j \leq d, i\not=j)$};
             \draw (1,-5.5) node {\textbf{(d)}};
        \end{tikzpicture}
\end{tabular} \caption{This figure illustrates the geometry of the statements for Theorem \ref{Thm1}, Corollary \ref{CorPHasExactlymRoots}, and Corollary \ref{cor2}. \textbf{(a)} The $\delta$-ball $B(b_i,\delta)$ centered at $b_i$, the $i$th coefficient of the polynomial $q(z)=b_nz^n+\cdots + b_1z+b_0$, contains the $i$th coefficient $a_i$ of the polynomial $p(z)=a_nz^n+\cdots+a_1z+a_0$, for each $i=0,\ldots, n$. \textbf{(b)} The $\varepsilon$-ball $B(\zeta,\varepsilon)$ with $\varepsilon>0$ given, is centered at a given root $\zeta$ of $q(z)$ with multiplicity $m$. Theorem \ref{Thm1} guarantees there exists $\delta>0$ such that if the coefficients of $p(z)$ and $q(z)$ are as in (a) then $B(\zeta,\varepsilon)$ contains $m$ roots of $p(z)$, which we denote by $\lambda_1,\ldots, \lambda_m$, but could possibly contain one or more of its remaining roots $\lambda_k$ for $m<k\leq \deg p$. \textbf{(c)} Corollary \ref{CorPHasExactlymRoots} guarantees in the case $\deg q=m$ or in the case $\deg q>m$, after constraining $\varepsilon$ so that $0<\varepsilon <\min_{1\leq j\leq d, \zeta_j\not=\zeta}|\zeta_j-\zeta| $, and possibly with a smaller $\delta>0$ from Theorem \ref{Thm1}, that there is exactly $m$ roots of $p(z)$, which we again denote by $\lambda_1, \ldots, \lambda_m$, that are in $B(\zeta,\varepsilon)$, while the remaining roots $\lambda_k$ for $m<k\leq \deg p$ of $p(z)$ lie outside this ball. \textbf{(d)} Corollary \ref{cor2} guarantees in the case $1<d$ that after constraining $\varepsilon$ further so that $0<\varepsilon <\min_{1\leq i,j\leq d, i\not=j}|\zeta_i-\zeta_j| $ and possibly with a smaller $\delta>0$ from Corollary \ref{CorPHasExactlymRoots}, that for each root $\zeta_j$ of $q(z)$ with multiplicity $m_j$ there is exactly $m_j$ roots of $p(z)$, which we denote by $\lambda_{j,1}, \ldots, \lambda_{j,m_j}$ in $B(\zeta_j,\varepsilon)$ and all the rest of its roots lie exterior to it. [Figure 1.(d) illustrates the case that $\varepsilon<\frac{1}{2}\min_{1\leq i,j\leq d, \zeta_j\not=\zeta_i}|\zeta_i-\zeta_j|$ and hence $B(\zeta_i,\varepsilon)\cap B(\zeta_j,\varepsilon)=\emptyset$ for all $i\not= j$.]} \label{Fig1}
\end{figure}
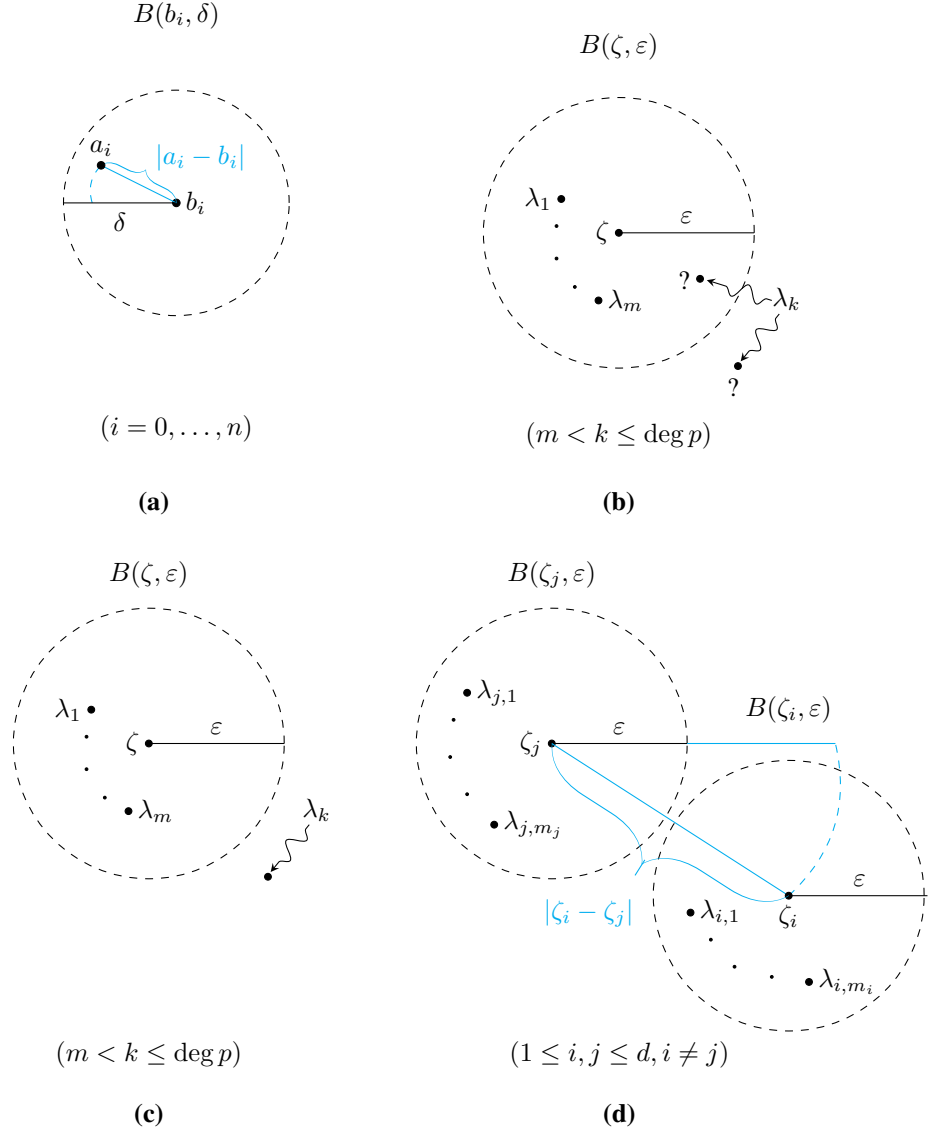

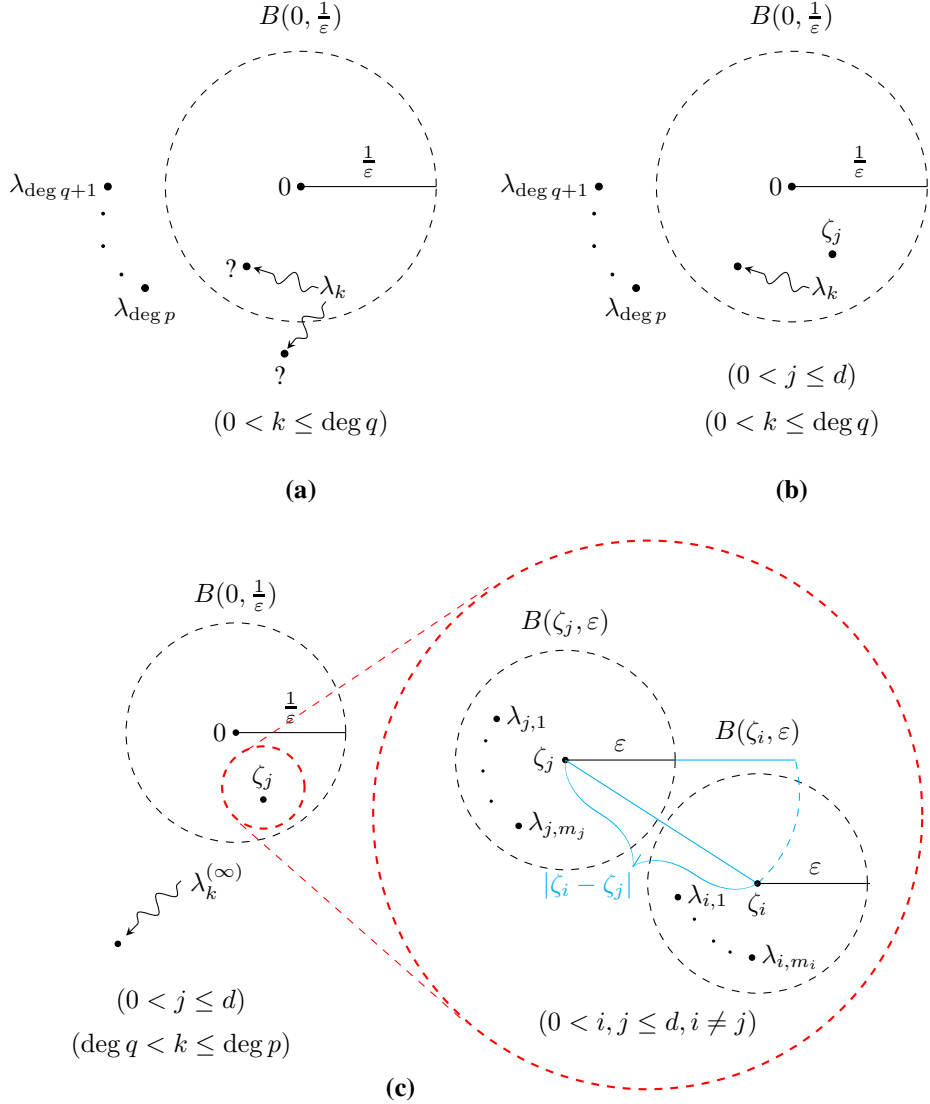
\begin{figure}[!]
\begin{tabular}{cc}
\hspace{-7.5cm}
        \begin{tikzpicture}[scale=0.9]
            \draw[dashed] (8,0.5)circle[radius=2];
            \draw (8,0.5) node[left] {$0$} -- (10,0.5);
            \draw [color=black, fill=black] (8,0.5)circle(0.05);
            \draw (9,0.5) node[above] {$\frac{1}{\varepsilon}$};
            \draw [color=black, fill=black] (5.7,-1)circle(0.05) node[below]{$\lambda_{\deg p}$};
            \draw [color=black, fill=black] (5.15,0.5)circle(0.05) node[left]{$\lambda_{\deg q +1}$};
            \draw [color=black, fill=black] (5.35,-0.8)circle(0.02);
            \draw [color=black, fill=black] (5.08,-0.38)circle(0.02);
            \draw [color=black, fill=black] (5.08,0.1)circle(0.02);
            \draw[swiggle] {(8.37 ,-1.2 ) ++ (0,0)} node[yshift=0.2cm, xshift=0.1cm] {$\lambda_{k}$} -- (7.8,-1.9) node[below, yshift=-0.1cm, xshift=-0.1cm] {?};
            \draw[swiggle] {(8.26 ,-1 ) ++ (0,0)} -- (7.3,-0.7) node[left, xshift=-0.1cm] {?};
            \draw [color=black, fill=black] (7.76,-1.97)circle(0.05);
            \draw [color=black, fill=black] (7.2,-0.68)circle(0.05);
            \draw (8,-3) node {$(0 < k \leq \deg q)$};
            \draw (8,3) node {$B(0,\frac{1}{\varepsilon})$};
            \draw (8,-4) node {\textbf{(a)}};
        \end{tikzpicture} \vspace{0.25cm}
& \hspace{-6.25cm}
     \begin{tikzpicture}[scale=0.9]
            \draw[dashed] (8,0.5)circle[radius=2];
            \draw (8,0.5) node[left] {$0$} -- (10,0.5);
            \draw [color=black, fill=black] (8,0.5)circle(0.05);
            \draw (9,0.5) node[above] {$\frac{1}{\varepsilon}$};
            \draw [color=black, fill=black] (5.7,-1)circle(0.05) node[below]{$\lambda_{\deg p }$};
            \draw [color=black, fill=black] (5.15,0.5)circle(0.05) node[left]{$\lambda_{\deg q+1}$};
            \draw [color=black, fill=black] (8.6,-0.5)circle(0.05) node[above]{$\zeta_{j}$};
            \draw [color=black, fill=black] (5.35,-0.8)circle(0.02);
            \draw [color=black, fill=black] (5.08,-0.38)circle(0.02);
            \draw [color=black, fill=black] (5.08,0.1)circle(0.02);
            \draw[swiggle] {(8.26 ,-1 ) node[right, xshift=-0.1cm] {$\lambda_{k}$} ++ (0,0)} -- (7.3,-0.7);
            \draw [color=black, fill=black] (7.2,-0.68)circle(0.05);
            \draw (8,-2.3) node {$(0 < j \leq d)$};
            \draw (8,-3) node {$(0 < k \leq \deg q)$};
            \draw (8,3) node {$B(0,\frac{1}{\varepsilon})$};
            \draw (8,-4) node {\textbf{(b)}};
        \end{tikzpicture}
  \\ \hspace{-0.3cm}
          \begin{tikzpicture}[scale=0.73]
            \draw[dashed] (2,0.5)circle[radius=2];
            \draw (4,0.5) -- (2,0.5)node[left] {$0$};
            \draw [color=black, fill=black] (2,0.5)circle(0.05);
            \draw (3,0.5) node[above] {$\frac{1}{\varepsilon}$};
            \draw[-stealth,decorate,decoration={snake,amplitude=3pt,pre length=2pt,post length=3pt}] (1,-2.2) node[right]{$\lambda_{k}^{(\infty)}$} -- ++(-1,-1);
            \draw [color=black, fill=black] (-0.15,-3.35)circle(0.05);
           \draw [color=black, fill=black] (2.5,-0.72)circle(0.05) node[above]{$\zeta_{j}$};
           \draw[dashed,red, thick] (2.5,-0.5)circle[radius=0.75];
            \draw (1,-4.4) node {$(0 < j \leq d)$};
            \draw (1,-5.2) node {$(\deg q < k \leq \deg p)$};
            \draw (2,3) node {$B(0,\frac{1}{\varepsilon})$};
            \draw (5,-6) node {\textbf{(c)}};
\draw[dashed,red, thick] (9.5,-1)circle[radius=5];
        \draw[dashed, red]  (2.08,-1.16) -- (6.2,-4.8);
        \draw[dashed, red]  (2.2,0.18) -- (6.6,3.15);
            \draw (8,2.5) node {$B(\zeta_j,\varepsilon)$};
            \draw[dashed] [color=black] (8,0) circle (2);
            \draw (8,0) node[left] {$\zeta_j$};
            \draw [color=black, fill=black] (8,0)circle(0.05);
            \draw (9,0) node[above] {$\varepsilon$};
            \draw (8,0) -- (10,0);
            \draw [color=black, fill=black] (7.15,-1.2)circle(0.05) node[right]{$\lambda_{j,m_j}$};
            \draw [color=black, fill=black] (6.75,0.75)circle(0.05) node[right]{$\lambda_{j,1}$};
            \draw [color=black, fill=black] (6.55,0.35)circle(0.02);
            \draw [color=black, fill=black] (6.5,-0.2)circle(0.02);
            \draw [color=black, fill=black] (6.65,-0.75)circle(0.02);
            \draw[dashed] (11.5,-2.25)circle[radius=2];
            \draw (11.5,0.5) node {$B(\zeta_i,\varepsilon)$};
            \draw (9.4,-2.3) node[left] {\textcolor{cyan}{$|\zeta_i-\zeta_j|$}};
            \draw (11.5,-2.25) node[below] {$\zeta_i$} -- (13.55,-2.25);
            \draw[cyan] (11.5,-2.25) -- (8,0);
            \path [cyan,bend right, dashed]   (11.5,-2.25) edge (12.19,0);
            \draw[cyan] (10,0) -- (12.19,0);
            \draw [color=black, fill=black] (11.5,-2.25)circle(0.05);
            \draw (12.55,-2.25) node[above] {$\varepsilon$};
            \draw [color=black, fill=black] (11.4,-3.6)circle(0.05) node[right]{$\lambda_{i,m_i}$};
            \draw [color=black, fill=black] (10.05,-2.5)circle(0.05) node[right]{$\lambda_{i,1}$};
            \draw [color=black, fill=black] (10.35,-2.9)circle(0.02);
            \draw [color=black, fill=black] (10.7,-3.3)circle(0.02);
            \draw [color=black, fill=black] (11,-3.5)circle(0.02);
            \draw [pen colour={cyan}, decorate, decoration = {calligraphic brace, amplitude=20pt}] (11.5,-2.25) -- (8,0);
            \draw[color=cyan] (9.375,-1.72) -- (9.23,-1.95);
            \draw (9.5,-4.8) node {$(0< i,j \leq d, i\not= j)$};
        \end{tikzpicture}
\end{tabular}
\caption{This figure illustrates the geometry of the statements for Proposition \ref{Prop1AtLeastZeroRoots}, Corollary \ref{Cor3HasExactlymRoots}, and Theorem \ref{ThmRootSepincludingInfty}. \textbf{(a)} The $\frac{1}{\varepsilon}$-ball $B(0,\frac{1}{\varepsilon})$, is centered at $0$, with $\varepsilon>0$ given. Proposition \ref{Prop1AtLeastZeroRoots} guarantees there exists $\delta>0$ such that if the coefficients of $p(z)$ and $q(z)$ are as in Fig.\ 1(a) then $B(0,\frac{1}{\varepsilon})$ excludes $\deg p-\deg q$ roots of $p(z)$, which we denote by $\lambda_{q+1},\ldots, \lambda_{\deg p}$ (if $\deg p>\deg q$), but could possibly exclude one or more of its remaining roots $\lambda_{k}$ for $0<k\leq \deg q$. \textbf{(b)} Corollary \ref{Cor3HasExactlymRoots} guarantees in the case $\deg q=0$ or in the case $\deg q>0$, after constraining $\varepsilon$ so that $0<\varepsilon <(\max_{1 \leq j \leq d}|\zeta_j|)^{-1}$ [or equivalently $\zeta_j\in B(0, \frac{1}{\varepsilon})$, for all $j$], and possibly with a smaller $\delta>0$ from Proposition \ref{Prop1AtLeastZeroRoots}, that there are exactly $\deg p-\deg q$ roots of $p(z)$ outside of $B(0,\frac{1}{\varepsilon})$, which we denote by $\lambda_k=\lambda_k^{(\infty)}$ for $\deg q<k\leq \deg p$, while the remaining roots $\lambda_{k}$ for $0<k\leq \deg q$ lie in the closure of this ball $\overline{B(0, \frac{1}{\varepsilon})}$ . \textbf{(c)} Theorem \ref{ThmRootSepincludingInfty} guarantees, after further constraining $\varepsilon$ [as in (\ref{BestConstrianedVarepsilon})], that there exists possibly a smaller $\delta>0$ such that the conclusions of both Corollary \ref{cor2} and Corollary \ref{Cor3HasExactlymRoots} are simultaneously true. The figure here illustrates this by essentially superimposing Fig.\ 1(d) onto Fig.\ 2(b) [i.e., the zooming out of the red dashed circle and its interior region about the root $\zeta_j$ of $q(z)$].}\label{Fig2}
\end{figure}

\begin{figure}[!t]
\begin{center}
\hspace{-0.5cm}
\begin{tikzpicture}[scale=1, transform shape,
blockone/.style ={rectangle, draw, text width=7.75em,align=center, minimum height=0.25em},
Cir/.style ={draw, circle, inner sep=2pt}
]
\draw (1.6,-0.75) node[text width=6cm] (q) {\begin{equation*}
    q(z)= b_{\deg q}\prod_{i>0}^{d}(z-\zeta_i)^{m_i} \qquad\text{\small($\zeta_i\not=\zeta_j$, if $i\not=j$)}
\end{equation*}};
\draw (4,-3.5) node[blockone, scale=0.8] (q) {\vspace{-0.4cm}\begin{align*}
   \infty:\lambda_k = \lambda_k^{(\infty)}
\end{align*}\vspace{-1cm}\begin{align*}
    \deg q < k \leq \deg p
\end{align*}};
\draw (1,-2.5) node[blockone, scale=0.8] (q) {\vspace{-0.4cm}\begin{align*}
   \zeta_i :\lambda_k = \lambda_{i, \ell}
\end{align*}\vspace{-1cm}\begin{align*}
   0< i \leq d
\end{align*}};
\draw (-2.8,-4.6) node[text width=6cm] (q) {\begin{equation*}
    p(z)= a_{\deg p}\prod_{k=1}^{\deg p}(z\;-\;\lambda_k)\;=\; a_{\deg p} \left[\prod_{i>0}^{d}\prod_{\ell =1}^{m_i}(z-\lambda_{i,\ell})\right] \prod_{k> \deg q}^{\deg p}(z\;-\;\lambda_{k}^{(\infty)})
\end{equation*}};
\draw (-2.35,-4.4) -- (-2.35,-3); 
\draw (-2.35,-3) -- (-2.28,-3); 
\draw[->, >=stealth] (1.7,-3.1) -- (1.7,-4.4); 
\draw[->, >=stealth] (-0.95,-2.5) -- (-0.25,-2.5); 
\draw (-1.5,-3) node[Cir] {\shortstack{\small root \\ \small grouping}};
\draw (1.7,-1.3) -- (1.7,-1.9); 
\draw[dashed, ->, >=stealth] (-0.95,-3.5) -- (2.75,-3.5);
\draw[dashed, ->, >=stealth] (4.55,-4.1) -- (4.55,-4.4);
\draw (0,-6.5) node[text width=6cm] (q) {\begin{equation*}
   B\left(0,\frac{1}{\varepsilon}\right) \ni \underset{i \not = j}{\zeta_j} \not \in B\left(\zeta_i , \varepsilon\right) \ni \lambda_{i, \ell} \not \in \mathbb{C} \setminus \overline{ B\left(0, \frac{1}{\varepsilon}\right)} \ni \lambda_k^{(\infty)}
\end{equation*}};
\draw[dashed, ->, >=stealth] (5.1,-5.6) -- (5.1,-6.3); 
\draw[->, >=stealth] (1.7,-5.1) -- (1.7,-6.3); 
\draw[dashed] (4.55,-5.1) |- (5.1,-5.6); 
\draw[->, >=stealth] (1.7,-1.69) -- (1.7,-1.7); 
\draw[->, >=stealth] (-2.35,-4.01) -- (-2.35,-4); 
\end{tikzpicture}
\end{center}
\caption{This figure illustrates Remark \ref{RemRootGroupings} on Theorem \ref{ThmRootSepincludingInfty} [cf.\ Figure \ref{Fig2}.(c)] on the separation of all the roots, $\lambda_k,$ $1\leq k\leq \deg p$, of the polynomial $p$ (counting multiplicities) into the group $\lambda_{i, \ell},$ $1\leq \ell\leq m_i$ corresponding to the distinct roots $\zeta_i$ of $q$ (if any) of multiplicity $m_i$ for $0< i\leq d,$ and into the group $\lambda_k^{(\infty)}$, $\deg q<k\leq \deg p$ corresponding to those roots of $p$ in $\mathbb{C}\setminus \overline{B\left(0,\frac{1}{\varepsilon}\right)}=\left\{z\in \mathbb{C}:|z|>\frac{1}{\varepsilon}\right\}$ (an open neighborhood of $\infty$).}\label{FigRemarkConclusion}
\end{figure}

\begin{corollary} \label{CorPHasExactlymRoots}
Suppose $\zeta$ is a root of $q(z)$ of multiplicity $m$ and
$\varepsilon>0$, if $m=\deg q$, or
\begin{align}
    0<\varepsilon <\min_{1 \leq j \leq d, \zeta_j \not = \zeta}|\zeta_j - \zeta|,
\end{align} 
if $m\not=\deg q$, where $\zeta_j$, for $j=1,\ldots, d$, denotes all the distinct roots of $q$. Then there exists a $\delta>0$ such that if $\max_{0\leq i\leq n}|a_i-b_i|<\delta$, then $p$ has exactly $m$ roots whose distance is less than $\varepsilon$ from $\zeta$.
\end{corollary}

\begin{proof}
Let $\zeta_1, \ldots, \zeta_d$ be all the distinct roots of $q$, $m_i$ be the multiplicity of the root $\zeta_i$ of $q$ for each $i=1,\ldots,d$, and $\zeta$ be any root of $q$. In particular, $\zeta=\zeta_{\ell}$ and $m=m_{\ell}$ for some $\ell$ with $1 \leq \ell \leq d$. Choose any $\varepsilon>0$ and, if $m\not=\deg q$, with $\varepsilon < \min_{1 \leq i \leq d, i \not = \ell} |\zeta_i - \zeta|$. Now consider the structure of the proof of Theorem \ref{Thm1}. Proceeding in a similar contradictory fashion as that proof with the only exception being that the $p_{k_j}$ will now have at least $m+1$ roots (hence $N=\deg p_k\geq m+1$) whose distance from $\zeta$ is less than $\varepsilon$ and again all its roots (counting multiplicity) are $\lambda_i^{(k_j)}$ which converge to $\mu_i$ (in $\mathbb{C}\cup \{\infty\}$) for $i=1,\ldots, N$. Again, it follows by this along with (\ref{SumDegreeFactoriization}) and (\ref{convervingfrompp'toqq'}) that
\begin{align}
    1=\lim_{z\to \zeta}\lim_{j\to \infty} \frac{(z-\zeta)\frac{p_{k_j}'(z)}{p_{k_j}(z)}}{(z-\zeta) \frac{q'(z)}{q(z)}}=\frac{\# \text{ of } i\text{'s such that }\mu_i=\zeta}{m}.
\end{align}
Thus, without loss of generality, we may assume the roots $\lambda_i^{(k_j)}$ are ordered such that
\begin{align}
    |\lambda_i^{(k_j)}-\zeta|<\varepsilon
\end{align}
for $i=1,\ldots, m+1$, $\lim_{j\to \infty}\lambda_i^{(k_j)}=\mu_i=\zeta$ for $i=1,\ldots, m$, $\lim_{j\to \infty}\lambda_{m+1}^{(k_j)}=\mu_{m+1}\in\mathbb{C}$, and $\mu_{m+1}\not =\zeta$. But we also have
\begin{align}
    &\lim_{z\to \mu_{m+1}}(z-\mu_{m+1}) \frac{q'(z)}{q(z)}=\left\{\begin{array}{ll}
    0, \text{ if }q(\mu_{m+1})\neq 0,\\
    m_j, \text{ if }\mu_{m+1}=\zeta_j\text{ for some }j.
    \end{array}
    \right.\\
    &=\lim_{z\to \mu_{m+1}}\lim_{j\to \infty} (z-\mu_{m+1})\frac{p_{k_j}'(z)}{p_{k_j}(z)}
    =(\# \text{ of } i\text{'s such that }\mu_i=\mu_{m+1})\geq 1,
\end{align}
implying that $\mu_{m+1}$ is a root of $q(z)$ with 
\begin{align}
   \min_{1 \leq i \leq d, i \not = \ell} |\zeta_i - \zeta|> \varepsilon \geq \lim_{j\to \infty}|\lambda_{m+1}^{(k_j)}-\zeta| = |\mu_{m+1}-\zeta_{\ell}| 
\end{align}
from which we conclude that we must have
\begin{align}
    \mu_{m+1}=\zeta,
\end{align}
a contradiction that $\mu_{m+1}\not = \zeta$. This proves the corollary.
\end{proof}

\begin{corollary}[Root separation without $\boldsymbol{\infty}$] \label{cor2}
Suppose $\zeta_1, \ldots, \zeta_d$ are all the distinct roots of $q(z)$ with the root $\zeta_j$ having multiplicity $m_j$, respectively, for $j=1,\ldots, d$. If $\varepsilon>0$ or, in the case $d>1$, \begin{align*} 0 < \varepsilon < \min_{1\leq i,j\leq d, i\not=j}|\zeta_i - \zeta_j|,\end{align*} then there exists a $\delta>0$ such that if $\max_{0\leq i\leq n}|a_i-b_i|<\delta$, then $p$ has exactly $m_j$ roots whose distance is less than $\varepsilon$ from $\zeta_j$ for each $j=1,\ldots,d$.
\end{corollary}
\begin{proof}
If $d=1$ then the statement is true by the previous corollary. Hence, suppose $d>1$. First, we have
\begin{align}
    \min_{1\leq i,j\leq d, i\not=j}|\zeta_i - \zeta_j|=\min_{1\leq j\leq d} \min_{1\leq i\leq d, i\not=j}|\zeta_i - \zeta_j|.
\end{align}
Now choose an $\varepsilon>0$ such that $\varepsilon<\min_{1\leq i,j\leq d, i\not=j}|\zeta_i - \zeta_j|$. Then, for each $j=1,\ldots, d$, it follows that 
\begin{align}
    0<\varepsilon <\min_{1\leq i\leq d, i\not=j}|\zeta_i - \zeta_j|
\end{align}
and thus by the previous corollary, there exists a $\delta_j>0$ such that if $\max_{0\leq i\leq n}|a_i-b_i|<\delta_j$, then $p$ has exactly $m_j$ roots whose distance is less than $\varepsilon$ from $\zeta_j$. Thus, if we take
\begin{align}
    \delta = \min_{1\leq j\leq d}\delta_j
\end{align}
then $0<\delta\leq \delta_j$ for each $j=1,\ldots, d$ and it follows that if $\max_{0\leq i\leq n}|a_i-b_i|<\delta$, then $p$ has exactly $m_j$ roots whose distance is less than $\varepsilon$ from $\zeta_j$ for all $j=1,\ldots, d$. This proves the corollary.
\end{proof}

At this point, the reader may find it helpful to look at Figure \ref{Fig1} which illustrates the statements in Theorem \ref{Thm1}, Corollary \ref{CorPHasExactlymRoots}, and Corollary \ref{cor2}.

Next, we will need the following definition and lemma (proof is left to the reader) in the next proposition.
\begin{definition}[Reciprocal polynomial]\label{DefReciprocalPoly}
For nonzero polynomials $p(z), q(z)$ with the form (\ref{Def_Poly_p}), (\ref{Def_Poly_q}), their reciprocal polynomials (or reflected polynomials) $\Tilde{p}(z), \Tilde{q}(z)$ are defined as 
\begin{align}
     \Tilde{p}(z)&=z^{\deg p}p\left (\frac{1}{z}\right )= a_0z^{\deg p} + a_{1}z^{\deg p-1} + \cdots + a_{\deg p-1}z+a_{\deg p} \label{Def_RPoly_p},\\
     \Tilde{q}(z)&=z^{\deg q}q\left (\frac{1}{z}\right )= b_0z^{\deg q} + b_{1}z^{\deg q-1} + \cdots + b_{\deg q-1}z+b_{\deg q} \label{Def_RPoly_q}.
\end{align}
\end{definition}
\begin{lemma}[Basic properties of reciprocal polynomials]\label{LemReciprocalPolyRootMultiplicities}
Let $p(z)$ be a nonzero polynomial of the form (\ref{Def_Poly_p}). Then 
\begin{align}
    \deg \Tilde{p}=\deg p-\min \{i:a_i\not = 0\}\leq \max \{i:a_i\not = 0\}=\deg p,
\end{align}
 \begin{align}
     \Tilde{p}(z)=a_{\deg \Tilde{p}}z^{\deg \Tilde{p}} + a_{\deg \Tilde{p}+1}z^{\deg \Tilde{p}-1} + \cdots + a_{\deg p},\;\; a_{\deg \Tilde{p}}\not = 0,
 \end{align}
 and
 \begin{align}
     \Tilde{p}(0)=a_{\deg p}\not =0.
 \end{align}
Furthermore, $\lambda$ is a nonzero root of $p(z)$ if and only if $\frac{1}{\lambda}$ is a root of $\Tilde{p}(z)$, in which case, the roots have the same multiplicities. Moreover, $p(z)$ and $\Tilde{p}(z)$ have the factorizations 
\begin{align}
    p(z)=a_{\deg p}\prod_{i=1}^{\deg p}(z-\lambda_i),\;\;\Tilde{p}(z)=a_{\deg \Tilde{p}}\prod_{i=1}^{\deg \Tilde{p}}(z-\Tilde{\lambda}_i),
\end{align}
where
\begin{align}
    \Tilde{\lambda}_i&=\frac{1}{\lambda_{i}},\;\;0<i\leq \deg {\Tilde{p}},\\
    \lambda_i&=0,\;\;\deg {\Tilde{p}}<i\leq \deg p.
\end{align}
\end{lemma}

\begin{proposition} \label{Prop1AtLeastZeroRoots}
 Assume $q(z)$ is a nonzero polynomial. If $\varepsilon >0$ then there exists a $\delta >0$ such that if $\max_{0\leq i\leq n}|a_i-b_i|<\delta$, then $p(z)$ has at least $\deg p-\deg q$ roots whose distance from $0$ is greater than $\frac{1}{\varepsilon}$.
\end{proposition}

\begin{proof} Assume that $q\not =0$, i.e., $q(z)$ is not the zero polynomial. Then we know that if $p$ is a polynomial such that $|a_{\deg q}-b_{\deg q}|<|b_{\deg q}|$ then $a_{\deg q}\not =0$ implying $\deg p\geq \deg q$. Thus, without loss of generality we can consider such polynomial $p$. Now consider the polynomial
\begin{align}
    z^{\deg p - \deg q}\Tilde{q}(z)=b_0z^{\deg p} + b_{1}z^{\deg p-1} + \cdots +b_{\deg q}z^{\deg p-\deg q}.\label{DefModifiedReciprocalPolyForq}
\end{align}
It has $\zeta = 0$ as a root of multiplicity $\deg p-\deg q\geq 0$. Let $\varepsilon >0$ be given. By Theorem \ref{Thm1}, there exists a $\delta >0$ such that if $\max_{0\leq i\leq n}|a_i-b_i|<\delta$ then $\Tilde{p}(z)$ has at least $\deg p-\deg q$ roots whose distance from $0$ is less than $\varepsilon$, that is, if we denote these roots by $\Tilde{\lambda}_i$ for $0<i\leq \deg p-\deg q$ then 
\begin{align}
\prod_{0<i\leq \deg p-\deg q}(z-\Tilde{\lambda}_i)\;|\;\Tilde{p}(z)\label{RefTildeLinearFactorsDivideReciprocalPoly}
\end{align} 
and $|\Tilde{\lambda}_i|<\varepsilon$ for $0<i\leq \deg p-\deg q$. Now since $\Tilde{p}(0)=a_{\deg p}\not = 0$ then $\Tilde{\lambda}_i\not =0$ for $0<i\leq \deg p-\deg q$ which implies that $\lambda_i=\frac{1}{\Tilde{\lambda}_i}$ is a root of $p(z)$ with $\varepsilon^{-1}<|\lambda_i|$ for $0<i\leq \deg p-\deg q$ and since (\ref{RefTildeLinearFactorsDivideReciprocalPoly}), it follows by Lemma \ref{LemReciprocalPolyRootMultiplicities} that $\prod_{0<i\leq \deg p-\deg q}(z-\lambda_i)\;|\;p(z)$. This proves the proposition.
\end{proof}

\begin{corollary}\label{Cor3HasExactlymRoots} Assume $q(z)$ is a nonzero polynomial and, if $\deg q\geq 1$, let $\zeta_1,\ldots, \zeta_d$ denote all its distinct roots. Let $\varepsilon$ be given such that \begin{align}\label{BestConstrianedVarepsilon} 
0<\varepsilon <\left\{ 
\begin{array}{cc}
\infty  & \text{if }\deg q=0 \text{ or if } d=1 \text{ and }\zeta_1=0,\\
\left(\max_{1 \leq j \leq d}|\zeta_j|\right)^{-1} & \text{otherwise.}
\end{array}
\right.
\end{align} 
Then there exists a $\delta>0$ such that if $\max_{0\leq i\leq n}|a_i-b_i|<\delta$, then $p$ has exactly $\deg p - \deg q$ roots whose distance from $0$ is greater than $\frac{1}{\varepsilon}$.
\end{corollary}
\begin{proof}
Assume $q(z)$ is a nonzero polynomial. If $\deg q=0$ then the statement follows immediately from the previous proposition. Hence suppose that $\deg q\geq 1$. First, apply Corollary \ref{CorPHasExactlymRoots} to $z^{\deg p - \deg q}\Tilde{q}(z)$ and $\Tilde{p}(z)$ for the root $0$ of $z^{\deg p - \deg q}\Tilde{q}(z)$. This will get us a $\delta>0$ such that if $0<\varepsilon <(\max_{1 \leq j \leq d}|\zeta_j|)^{-1} (=\infty \text{ if $d=1$ and $\zeta_1=0$})$, then if  $\max_{0\leq i\leq n}|a_i-b_i|<\delta$, then $\Tilde{p}(z)$ has exactly $\deg p - \deg q$ roots whose distance from $0$ is less than $\varepsilon$ and none of these roots are zero since $\Tilde{p}(0)\not =0$. This means that $p(z)$ has at least $\deg p - \deg q$ nonzero roots whose distance from $0$ is greater than $\varepsilon^{-1}$. But it cannot have more than $\deg p - \deg q$ roots whose distance from $0$ is greater than $\varepsilon^{-1}$ since if it did then they would all have to be nonzero and hence you would have more than $\deg p - \deg q$ roots of $\Tilde{p}(z)$ whose distance from $0$ is less than $\varepsilon$, a contradiction.
\end{proof}

\begin{theorem}[Root separation including $\boldsymbol{\infty}$] \label{ThmRootSepincludingInfty}
Suppose $q(z)$ is a nonzero polynomial and, if $\deg q\geq 1$, let $\zeta_1, \ldots, \zeta_d$ be all the distinct roots of $q(z)$ with the root $\zeta_j$ having multiplicity $m_j$, respectively, for $j=1,\ldots, d$. Let $\varepsilon$ be given such that \begin{align}
0<\varepsilon <\min\left\{\omega, \psi \right\},
\end{align} 
where we define 
\begin{align}
    \omega =\min_{1\leq i,j\leq d, i\not=j}|\zeta_i - \zeta_j|,\;\;\psi = \left(\max_{1 \leq j \leq d}|\zeta_j|\right)^{-1},
\end{align}
with $\omega =\infty$, if $\deg q=0$ or $d=1$ and $\psi = \infty$, if $\deg q=0$ or $d=1$ and $\zeta_1=0$. Then there exists a $\delta>0$ such that if $\max_{0\leq i\leq n}|a_i-b_i|<\delta$, then $p$ has exactly $\deg p - \deg q$ roots whose distance from $0$ is greater than $\frac{1}{\varepsilon}$ and, if $\deg q\geq 1$, then $p$ has exactly $m_j$ roots whose distance from $\zeta_j$ is less than $\varepsilon$ for all $j=1,\ldots,d$.
\end{theorem}
\begin{proof}
This is an immediate corollary of Corollaries 2 and 3. Indeed, let $\varepsilon$ be chosen such that (\ref{BestConstrianedVarepsilon}) is true. Then by Corollary 2, if $\deg q\geq 1$, we know that there exists a $\delta_1>0$ such that if $\max_{0\leq i\leq n}|a_i-b_i|<\delta_1$ then $p$ has exactly $m_j$ roots whose distance from $\zeta_j$ is less than $\varepsilon$ for all $j=1,\ldots,d$. By Corollary 3 we know that there exists a $\delta_2>0$ such that if $\max_{0\leq i\leq n}|a_i-b_i|<\delta_2$ then $p$ has exactly $\deg p - \deg q$ roots whose distance from $0$ is greater than $\frac{1}{\varepsilon}$. Let
\begin{align}
    \delta = \left\{ 
\begin{array}{cc}
\delta_2 & \text{if }\deg q=0, \\ 
\min\{\delta_1, \delta_2\}& \text{otherwise.}
\end{array}
\right.
\end{align}
Then it follows that $\delta>0$ and if $\max_{0\leq i\leq n}|a_i-b_i|<\delta$ then $p$ has exactly $\deg p - \deg q$ roots whose distance from $0$ is greater than $\frac{1}{\varepsilon}$ and, if $\deg q \geq 1$, $p$ has exactly $m_j$ roots whose distance from $\zeta_j$ is less than $\varepsilon$ for all $j=1,\ldots,d$. This proves the corollary.
\end{proof}

At this point, the reader may find it helpful to look at Figure \ref{Fig2} which illustrates the statements in Proposition \ref{Prop1AtLeastZeroRoots}, Corollary \ref{Cor3HasExactlymRoots}, and Theorem \ref{ThmRootSepincludingInfty}.

\begin{remark}[The balls partition the roots of $\boldsymbol{q,p}$ into disjoint groups]\label{RemRootGroupings}
It should be remarked that, for the $\varepsilon$ in Theorem \ref{ThmRootSepincludingInfty}, one of the key geometric features is that for the collection of $\varepsilon$-balls $B(\zeta_j, \varepsilon)$ centered at the distinct roots $\zeta_j$ of the polynomial $q$ (if any) ($0< j\leq d$) we have $\zeta_j\not\in B(\zeta_i,\varepsilon)$ for $0< i,j\leq d$ with $i\neq j$. In addition, for the $\frac{1}{\varepsilon}$-ball $B(0,\frac{1}{\varepsilon})$ centered at $0$ we also have $\zeta_j\in B(0,\frac{1}{\varepsilon})$ for each $j$. Then, by Theorem \ref{ThmRootSepincludingInfty}, it follows that if $p$ is a polynomial near $q$ [in the sense of Theorem \ref{ThmRootSepincludingInfty}, i.e., that $\delta>0$ and $\max_{0\leq i\leq n}|a_i-b_i|<\delta$ in terms of the coefficients of $q,p$ in (\ref{Def_Poly_q}), (\ref{Def_Poly_p})], then denoting its roots (counting multiplicities) by $\lambda_k,1\leq k\leq \deg p$, it follows that these roots may be grouped as $\lambda_{i\ell}, 1\leq \ell \leq m_i, 0< i\leq d$ and $\lambda_k^{(\infty)},\deg q<k\leq \deg p$ as follows: for all $i,j$ such that $0< i,j\leq d$ with $i\neq j$ and all $\ell$ such that $1\leq \ell\leq m_i$ and all $k$ such that $\deg q<k\leq \deg p$ we have
\begin{align}
    B\left(0,\frac{1}{\varepsilon}\right)\ni \zeta_j\not\in B(\zeta_i,\varepsilon)\ni \lambda_{i,\ell}\not\in\mathbb{C}\setminus \overline{B\left(0,\frac{1}{\varepsilon}\right)}\ni \lambda_{k}^{(\infty)}.
\end{align}
This is illustrated in Figure \ref{Fig2}.(c) and Figure \ref{FigRemarkConclusion}.
\end{remark}

\begin{remark}[The case $\boldsymbol{q}$ is the zero polynomial, i.e., $\boldsymbol{q=0}$]
Our results cover any polynomial $q\in P_n(\mathbb{C})$ except the zero polynomial, i.e., the case $q=0$. In this case there is no theory comparable to the theory developed above as the following examples illustrate: The constant zero polynomial $q(z)=0\in P_n(\mathbb{C})$ has every $\zeta\in \mathbb{C}$ as a root, but no matter how small $|a|$, the constant polynomial $p(z)=a\in P_n(\mathbb{C})$ has no roots, if $a\neq 0$. Now let $\lambda_i\in \mathbb{C}$, for $i=1,\ldots, \ell$ with $\ell\in \mathbb{N}$ and $\ell\leq n$. Consider the polynomial $p(z)=a\prod_{i=1}^\ell (z-\lambda_i)\in P_n(\mathbb{C})$ with $0\neq a\in \mathbb{C}$. We are completely free to choose these roots $\lambda_i$ of $p$ in $\mathbb{C}$, yet the leading order coefficient $a$ of $p$ can vary independent of those roots and hence the coefficients of $p$ can be made as small as we like by making $|a|$ sufficiently small.
\end{remark}

\section{Quantitative Results}\label{SecQuantitativeResults}
The following theorem is the main result of this section. It can be considered a quantitative result that complements the qualitative results in Sec. \ref{Qualitative Section}. For instance, given such an $\varepsilon>0$ in the hypotheses of Theorem \ref{ThmRootSepincludingInfty}, our theorem below can be used to find such a $\delta>0$ so that the conclusion of that theorem is true (see also Remark \ref{RemGenThmQuant}). We will give a simple and elementary proof of this theorem using Proposition \ref{PropPrimPertBddMonicPolys}, which itself is proved with Lemma \ref{LemRootUpperBoundMonicPoly}, below. We also provide a concrete example (Example \ref{ExNumericalOfQuantitativeMainThm}) using these results.
\begin{theorem}\label{ThmQuantOstrowskiLike}
Let $p,q\in P_n(\mathbb{C})$ be polynomials having the form (\ref{Def_Poly_p}) and (\ref{Def_Poly_q}), respectively. Assume $q(0)\neq 0$. Let $\zeta_i$, $0<i\leq \deg q$ denote all the roots (counting multiplicities) of $q$, respectively. Suppose
\begin{align}
|a_0-b_0|\leq \frac{1}{2}|b_0|,\;\;|a_{\deg q}-b_{\deg q}|< |b_{\deg q}|,\;\;
    \max_{0\leq i\leq n}|a_i-b_i|\leq \sum_{i=0}^n\left|\frac{b_i}{b_0}\right|,\label{HypThmQuantOstrowskiLike}
\end{align}
and, in the case $\deg q\not =0$,
\begin{align}
   \max_{0\leq i\leq n}|a_i-b_i|\leq \frac{1}{C}\left(2\max_{0< i\leq \deg q}|\zeta_i|\right)^{-n},\label{Hyp2ThmQuantOstrowskiLike}
\end{align}
where $C$ is given by (\ref{C}). If $\lambda\in \mathbb{C}$ and
\begin{align}
    p(\lambda)=0
\end{align}
then
\begin{align}
    |\lambda|^{-1}\leq C^{1/n} \left(\max_{0\leq i\leq n}\left|a_i-b_i\right|\right)^{1/n} \label{infintyrootforexample}
\end{align}
or
\begin{align} \label{Rootestimateforexample1}
    \min_{0<i\leq \deg q}|\lambda-\zeta_i|\leq 2\left(\max_{0< i\leq \deg q}|\zeta_i|\right)^2C^{1/n} \left(\max_{0\leq i\leq n}\left|a_i-b_i\right|\right)^{1/n}.
\end{align}

\end{theorem}
\begin{proof}
First, the hypotheses (\ref{HypThmQuantOstrowskiLike}) imply that
\begin{align*}
    a_0\not =0,\;\;a_{\deg q}\neq 0
\end{align*}
and hence
\begin{align*}
    \deg p\geq \deg q.
\end{align*}
Next, since $\deg p-\deg q\geq 0$ with $p(0)=a_0\not =0$ and $0\not =q(0)=a_0$, we can define the degree $n$ polynomials $f,g\in P_n(\mathbb{C})$ in terms of the reciprocal polynomials of $p$ and $q$ by
\begin{align}
    f(z)=z^{n-\deg p}\Tilde{p}(z),\; g(z)=z^{n-\deg p}z^{\deg p-\deg q}\Tilde{q}(z)
\end{align}
in which the polynomials $\Tilde{p}(z)$ and $z^{\deg p-\deg q}\Tilde{q}(z)$ have the form (\ref{Def_RPoly_p}) and (\ref{DefModifiedReciprocalPolyForq}), respectively. Let $\zeta_i$, $0<i\leq \deg q$ denote all the roots (counting multiplicities) of $q$. Then 
\begin{align}
    \Tilde{\zeta}_i=\zeta_i^{-1}\qquad(0<i\leq \deg q),
\end{align}
are all the nonzero roots (counting multiplicities) of $g$ and
\begin{align}
    \Tilde{\zeta}_i=0\qquad(\deg q<i\leq n),
\end{align}
are all the zero roots (counting multiplicities) of $g$.  Suppose $p(\lambda)=0$. Then $\lambda\not =0$ and $f(\lambda^{-1})=0$ so by hypotheses (\ref{HypThmQuantOstrowskiLike}) and Proposition \ref{PropPrimPertBddMonicPolys} below it follows that
\begin{align}
    \min_{1\leq i\leq n}|\lambda^{-1}-\Tilde{\zeta}_i|\leq C^{1/n} \left(\max_{0\leq i\leq n}\left|a_i-b_i\right|\right)^{1/n},
\end{align}
with $C$ given by (\ref{C}). Thus either that minimum is at a $j$ with $\Tilde{\zeta}_j=0$ or with $\Tilde{\zeta}_j=\zeta_j^{-1}$ in which case either
\begin{align}
    |\lambda|^{-1}=\min_{1\leq i\leq n}|\lambda^{-1}-\Tilde{\zeta}_i|\leq C^{1/n} \left(\max_{0\leq i\leq n}\left|a_i-b_i\right|\right)^{1/n}
\end{align}
or
\begin{align}
    |\lambda^{-1}-\zeta_j^{-1}|=\min_{1\leq i\leq n}|\lambda^{-1}-\Tilde{\zeta}_i|\leq C^{1/n} \left(\max_{0\leq i\leq n}\left|a_i-b_i\right|\right)^{1/n}.
\end{align}
In the latter case, this implies by the reverse triangle inequality and hypothesis (\ref{Hyp2ThmQuantOstrowskiLike}) that
\begin{align}
   |\lambda|^{-1}&\geq |\zeta_j|^{-1}-C^{1/n} \left(\max_{0\leq i\leq n}\left|a_i-b_i\right|\right)^{1/n}\\
   &\geq \min_{0< i\leq \deg q}|\zeta_i|^{-1}-C^{1/n} \left(\max_{0\leq i\leq n}\left|a_i-b_i\right|\right)^{1/n}\\
   &=\left(\max_{0< i\leq \deg q}|\zeta_i|\right)^{-1}-C^{1/n} \left(\max_{0\leq i\leq n}\left|a_i-b_i\right|\right)^{1/n}\\
   &\geq\frac{1}{2}\left(\max_{0< i\leq \deg q}|\zeta_i|\right)^{-1}
\end{align}
and hence
\begin{align}
    |\lambda-\zeta_j|&=|\lambda||\zeta_j||\lambda^{-1}-\zeta_j^{-1}|\\
    &\leq |\lambda||\zeta_j|C^{1/n} \left(\max_{0\leq i\leq n}\left|a_i-b_i\right|\right)^{1/n}\\
    &\leq 2|\zeta_j|\max_{0< i\leq \deg q}|\zeta_i|C^{1/n} \left(\max_{0\leq i\leq n}\left|a_i-b_i\right|\right)^{1/n}\\
    &\leq 2\left(\max_{0< i\leq \deg q}|\zeta_i|\right)^2C^{1/n} \left(\max_{0\leq i\leq n}\left|a_i-b_i\right|\right)^{1/n}.
\end{align}
The proof of the theorem now follows.
\end{proof}

\begin{remark}\label{RemGenThmQuant}
Notice in Theorem \ref{ThmQuantOstrowskiLike} we have the hypothesis that $q(0)\not=0$. This hypothesis can be removed but the trade-off is a more complicated statement for Theorem \ref{ThmQuantOstrowskiLike}. Indeed, if $q(0)=0$ then the starting point to prove such a theorem in this case is to replace the polynomials $p(z),q(z)$ in Theorem \ref{ThmQuantOstrowskiLike} by the shifted polynomials $p(z+z_0),q(z+z_0)$ for any fixed $z_0$ such that $q(z_0)\neq 0$. The reader will find it to be a good exercise to formulate and prove a version of Theorem \ref{ThmQuantOstrowskiLike} in this case. 
\end{remark}

\begin{example}\label{ExNumericalOfQuantitativeMainThm}
We will apply Theorem \ref{ThmQuantOstrowskiLike} to the polynomials
\begin{align}
p(z)=\sum _{i=0}^4a_iz^i,\;\;q(z) = z^3-z^2+z-1=(z-1)(z-\sqrt{-1})(z+\sqrt{-1})
\end{align}
in $P_n(\mathbb{C})$ with $n=4$. Then the coefficients and roots of $q(z)$ are
\begin{align}
&b_4=0,\; b_{\deg q} =b_3 = 1,
\; b_2=-1,
\; b_1=1,
\; b_0 = q(0) = -1\not =0,\\
&\zeta _{1} =1, \; \zeta_2 =\sqrt{-1} ,\; \zeta_3=-\sqrt{-1}.
\end{align}
Next, we calculate 
\begin{align}
    &\sum_{i=0}^n\left|\frac{b_i}{b_0}\right|=4,\;\;D=\left[1+\sum_{i=0}^n\frac{2}{|b_0|^2}(|b_0|+|b_i|)\right]\sum_{i=0}^n\left|\frac{b_i}{b_0}\right|=76,\\
    C&=\sum_{i=1}^n\frac{2}{|b_0|^2}(|b_0|+|b_i|)D^{n-i}=1779314.
\end{align}
Then the hypotheses (\ref{HypThmQuantOstrowskiLike}) and (\ref{Hyp2ThmQuantOstrowskiLike}) in Theorem \ref{ThmQuantOstrowskiLike} in this example become:
\begin{align*}
     |a_0-b_0| &\leq \frac{1}{2},\;\;\left(\text{i.e., }|a_0-b_0| \leq \frac{1}{2} |b_0|\right)\\
     |a_3-b_3| &\leq 1,\;\;\left(\text{i.e., }|a_{\deg q}-b_{\deg q}| < |b_{\deg q}|\right) \\
     \max_{0\leq i\leq 4}|a_i-b_i| &\leq 4,\;\; \left(\text{i.e., }\max_{0\leq i\leq n}|a_i-b_i| \leq \sum_{i=0}^n\left|\frac{b_i}{b_0}\right|\right)\\
     \max_{0\leq i\leq 4}|a_i-b_i| &\leq \frac{1}{28469024}.\;\; \left(\text{i.e., }\max_{0\leq i\leq n}|a_i-b_i| \leq \frac{1}{C}\left(2\max_{0< i\leq \deg q}|\zeta_i|\right)^{-n} \right)
\end{align*}
From this it follows that the hypotheses of Theorem \ref{ThmQuantOstrowskiLike} are satisfied in this example if and only if the last constraint on the coefficients of $p(z)$ hold. In this case, $p(z)$ has four roots $\lambda_1,\lambda_2,\lambda_3,\lambda_4$ and by Theorem \ref{ThmQuantOstrowskiLike} each root $\lambda$ must satisfy the inequalities (\ref{infintyrootforexample}) and (\ref{Rootestimateforexample1}) which in this case becomes
\begin{align}
    |\lambda|^{-1}\leq C^{1/4} \left(\max_{0\leq i\leq 4}\left|a_i-b_i\right|\right)^{1/4}
\end{align}
or
\begin{align}
    \min_{1\leq i\leq 3}|\lambda-\zeta_i|\leq 2C^{1/4} \left(\max_{0\leq i\leq 4}\left|a_i-b_i\right|\right)^{1/4}.
\end{align}

For instance, all this is true for the whole family of polynomials
\begin{align}
    p(z)&=\eta z^4+(1+\eta)z^3+(-1+\eta)z^2+(1+\eta)z+(-1+\eta),
\end{align}
provided 
\begin{align}
    |\eta|\leq \frac{1}{28469024}
\end{align}
since in this case the coefficients of $ p(z)$ are just those of $q(z)$ perturbed by $\eta$, i.e.,
\begin{align}
    a_4=b_4+\eta,\;\;a_3=b_3+\eta,\;\;a_2=b_2+\eta,\;\;a_1=b_1+\eta,\;\;a_0=b_0+\eta,
\end{align}
and hence
\begin{align}
    \max_{0\leq i\leq 4}|a_i-b_i| =|\eta|\leq \frac{1}{28469024}.
\end{align}
Thus, it has four roots $\lambda_1,\lambda_2,\lambda_3,\lambda_4$, and each such root $\lambda$ must satisfy
\begin{align}
    |\lambda|^{-1}\leq C^{1/4} |\eta|^{1/4}\;\; \text{ or } \;\;  \min_{1\leq i\leq 3}|\lambda-\zeta_i|\leq 2C^{1/4}|\eta|^{1/4}.
\end{align}
For example, this is satisfied for the value $\eta=10^{-8}$ so consider the polynomial
\begin{align}
    p(z)&=10^{-8} z^4+(1+10^{-8})z^3+(-1+10^{-8})z^2+(1+10^{-8})z+(-1+10^{-8}).
\end{align}
Then a numerical calculation of its roots (using Wolfram $|$ Alpha$^{\circledR}$) yields the approximations 
\begin{gather*}  
\lambda_{1}=1.00000,\\
\lambda_{2}\approx 2. 5000\times 10^{-9}+1. 00000\sqrt{-1},\; \lambda_{3}\approx 2. 5000\times 10^{-9}-1. 00000\sqrt{-1},\\ \lambda_4\approx -1.0\times 10^8 \sqrt{-1}
\end{gather*}
and, according to our theory, each such root $\lambda$ must satisfy 
\begin{align}\label{firstboundappx}
    |\lambda|^{-1}\leq C^{1/4} |\eta|^{1/4}=(.01779314)^{1/4}\approx 0.36523
\end{align}
or
\begin{align}\label{secondboundappx}
    \min_{1\leq i\leq 3}|\lambda-\zeta_i|\leq 2C^{1/4}|\eta|^{1/4}=2(.01779314)^{1/4}\approx 0.73045.
\end{align}
Comparing this to our approximations of the roots, we find that they are in good agreement with our theory since
\begin{align*}
    |\lambda_1-\zeta_1|&=0,  |\lambda_2-\zeta_2|= |\lambda_3-\zeta_3|= 2. 5000\times 10^{-9}\leq 0.73045,\\
    |\lambda_4|^{-1}&=1.0\times 10^{-8}\leq 0.36523.
\end{align*} 
\end{example}

In the rest of this section we prove the remaining statements needed to complete the proof of Theorem \ref{ThmQuantOstrowskiLike}. The following elementary result, which gives an upper bound on all the roots of a monic polynomial, is from \cite[Theorem 1]{97HM} which we will prove here for completeness.
\begin{lemma}\label{LemRootUpperBoundMonicPoly}
Let $f\in P_n(\mathbb{C})$ be a monic polynomial of degree $n\geq 1$, i.e.,
\begin{align}\label{monicf}
    f(z)=a_0z^n+a_{1}z^{n-1}+\cdots+a_n=\sum_{i=0}^na_iz^{n-i},\;a_0=1.
\end{align}
If $\alpha\in \mathbb{C}$ and $f(\alpha)=0$ then 
\begin{align}\label{RootBound}
    |\alpha| \leq \max\left\{1, \sum_{i=1}^n |a_i|\right\}.
\end{align}
\end{lemma}
\begin{proof}
Suppose $f(\alpha)=0$. Then either $|\alpha|\leq 1$, in which case (\ref{RootBound}) is true, or $|\alpha|> 1$ in which case (\ref{RootBound}) is also true since by triangle inequality
\begin{align}
    |\alpha|=\left|\frac{f(\alpha) -\alpha^n}{\alpha^{n-1}}\right|\leq \sum_{i=1}^n |a_i||\alpha|^{1-i}\leq \sum_{i=1}^n |a_i|.
\end{align}
\end{proof}

The following proposition can give a quantitative result for Theorem \ref{ThmRootSepincludingInfty} that can be used when $\deg p=\deg q$.
\begin{proposition}\label{PropPrimPertBddMonicPolys}
Let $f,g\in P_n(\mathbb{C})$ be two polynomials with
\begin{align}
    f(z)&=a_0z^n+\cdots+a_n=\sum_{i=0}^na_iz^{n-i},\\
    g(z)&=b_0z^n+\cdots+b_n=\sum_{i=0}^nb_iz^{n-i},\;\;b_0\neq 0
\end{align}
such that
\begin{align}
|a_0-b_0|\leq \frac{1}{2}|b_0|,\;\;
   \max_{0\leq i\leq n}|a_i-b_i|\leq \sum_{i=0}^n\left|\frac{b_i}{b_0}\right|.\label{HypPropPrimPertBddMonicPolys}
\end{align}
Let $\beta_1,\ldots, \beta_n$ denote all the roots (counting multiplicities) of $g$. If $\alpha\in \mathbb{C}$ and 
\begin{align}
f(\alpha)=0   
\end{align}
then
\begin{align}\label{Prop2Estimate}
    \min_{1\leq i\leq n}|\alpha-\beta_i|\leq C^{1/n} \left(\max_{0\leq i\leq n}\left|a_i-b_i\right|\right)^{1/n},
\end{align}
where
\begin{align}\label{C}
    C&=\sum_{i=1}^n\frac{2}{|b_0|^2}(|b_0|+|b_i|)D^{n-i},\;\;D=\left[1+\sum_{i=0}^n\frac{2}{|b_0|^2}(|b_0|+|b_i|)\right]\sum_{i=0}^n\left|\frac{b_i}{b_0}\right|.
\end{align}

\end{proposition}
\begin{proof}
First, by the reverse triangle inequality and the hypothesis (\ref{HypPropPrimPertBddMonicPolys}) we have
\begin{align*}
    -|a_0|+|b_0|\leq|a_0-b_0|\leq \frac{1}{2}|b_0|
\end{align*}
which implies
\begin{align*}
    |a_0|\geq \frac{1}{2}|b_0|>0.
\end{align*}
From this and the triangle inequality it follows that
\begin{align*}
    \left|\frac{a_i}{a_0}-\frac{b_i}{b_0}\right|&=\frac{1}{|a_0||b_0|}|(a_i-b_i)b_0-(a_0-b_0)b_i|\leq \frac{2}{|b_0|^2}|(a_i-b_i)b_0-(a_0-b_0)b_i|\\ &\leq\frac{2}{|b_0|^2}(|a_i-b_i||b_0|+|a_0-b_0||b_i|)\leq \frac{2}{|b_0|^2}(|b_0|+|b_i|)\max_{0\leq i\leq n}|a_i-b_i|
\end{align*}
for each $i=0,\ldots, n$.
From this, the hypothesis (\ref{HypPropPrimPertBddMonicPolys}), and the triangle inequality, we have
\begin{align*}
    \max \left\{1,\sum_{i=1}^n\left|\frac{a_i}{a_0}\right|\right\}&\leq 
    \sum_{i=0}^n\left|\frac{a_i}{a_0}\right|\\
    &\leq\sum_{i=0}^n\left|\frac{a_i}{a_0}-\frac{b_i}{b_0}\right|+\sum_{i=0}^n\left|\frac{b_i}{b_0}\right|\\
    &\leq \sum_{i=0}^n\frac{2}{|b_0|^2}(|b_0|+|b_i|)\max_{0\leq i\leq n}|a_i-b_i|+\sum_{i=0}^n\left|\frac{b_i}{b_0}\right|\\
    &\leq \left[\sum_{i=0}^n\frac{2}{|b_0|^2}(|b_0|+|b_i|)\right]\sum_{i=0}^n\left|\frac{b_i}{b_0}\right|+\sum_{i=0}^n\left|\frac{b_i}{b_0}\right|\\
    &=\left[1+\sum_{i=0}^n\frac{2}{|b_0|^2}(|b_0|+|b_i|)\right]\sum_{i=0}^n\left|\frac{b_i}{b_0}\right|=D.
\end{align*}

Now suppose $f(\alpha)= 0$. Then, by these inequalities, the triangle inequality, and Lemma \ref{LemRootUpperBoundMonicPoly}, we have
\begin{align*}
\left(\min_{1\leq i\leq n}|\alpha - \beta_i|\right)^n&\leq \prod_{i=1}^n|\alpha - \beta_i|= \left|\frac{1}{b_0}g(\alpha)\right|  = \left|\frac{1}{a_0}f(\alpha)-\frac{1}{b_0}g(\alpha)\right|\\
&\leq \sum_{i=1}^n\left|\frac{a_i}{a_0}-\frac{b_i}{b_0}\right||\alpha|^{n-i}\\
&\leq \sum_{i=1}^n\left|\frac{a_i}{a_0}-\frac{b_i}{b_0}\right|\left(\max \left\{1,\sum_{j=1}^n\left|\frac{a_j}{a_0}\right|\right\}\right)^{n-i}\\
&\leq \max_{0\leq i\leq n}|a_i-b_i|\sum_{i=1}^n\frac{2}{|b_0|^2}(|b_0|+|b_i|)D^{n-i}\\
&= C \max_{0\leq i\leq n}|a_i-b_i|.
\end{align*}
The proof follows now immediately from this by taking the $n$th root.
\end{proof}

\section{Applications}\label{SecApplications}
In this section we will consider an application of our results above in stability theory of multivariate polynomials. An important result in this theory is Theorem \ref{11RPThm} below. Our goal here is to provide an elementary proof of it using just the results above on single variable polynomials, namely, Theorem \ref{Thm1}, and some basic properties of multivariate polynomials described below.

\subsection{Multivariate polynomials}
Let $\mathbb{C}[z]=\mathbb{C}[z_1,\ldots, z_m]$ denote the set of multivariate polynomials in the $m$-variables $z_1,\ldots, z_m$ [where $z=(z_1,\ldots,z_m)$] with coefficients in $\mathbb{C}$. Every $f(z)\in \mathbb{C}[z]$ can be written uniquely as a finite linear combination of monomials in these variables (equivalently, the monomials in $\mathbb{C}[z]$ form a basis for the vector space $\mathbb{C}[z]$ over the field $\mathbb{C}$). More precisely, a monomial $z^{\beta}$ in $\mathbb{C}[z]$ is defined as the finite product
\begin{align}
    z^{\beta}=\prod_{j=1}^mz_j^{\beta(j)},
\end{align}
where $\beta\in (\mathbb{N}\cup\{0\})^m$ denotes an $m$-tuple of nonnegative integers with $j$th component $\beta(j)$, for each $j=1,\ldots, m$. Hence, the set of all monomials in $\mathbb{C}[z]$ is
\begin{align}
    \{z^{\beta}:\beta\in (\mathbb{N}\cup\{0\})^m\}.
\end{align}
Thus, if $f(z)\in \mathbb{C}[z]$ then there exists unique complex numbers $a_{\beta}\in \mathbb{C}$, $\beta\in  (\mathbb{N}\cup\{0\})^m$ with $a_\beta = 0$ for all but a finite number of $\beta$ such that $f(z)$ is the finite sum
\begin{align}
    f(z)=\sum_{\beta \in (\mathbb{N}\cup\{0\})^m}a_{\beta}z^{\beta}\label{DefMultiPolyMonomLinearComb}
\end{align}
and, if $f$ is not the zero polynomials (i.e., there exists a $\beta$ such that $a_\beta\not =0$), then the degree of $z_j$ in $f$ is defined in terms of the representation (\ref{DefMultiPolyMonomLinearComb}) as
\begin{align}
    \deg_{z_j}f=\max\{\beta(j):a_{\beta}\not=0\},\label{DefDegreezjInf}
\end{align}
for each $j=1,\ldots, m$.

In particular, notice that since monomials $z^{\beta}$ treated as functions $z^{\beta}:\mathbb{C}^m\rightarrow \mathbb{C}$ are continuous on $\mathbb{C}^m$ (because finite products of continuous functions are continuous and each $z_i:\mathbb{C}^m\rightarrow \mathbb{C}$ is continuous on $\mathbb{C}^m$) then every polynomial $f(z)\in \mathbb{C}[z]$ treated as a function $f:\mathbb{C}^m\rightarrow \mathbb{C}$ is continuous on $\mathbb{C}^m$ (since finite linear combinations of continuous functions are continuous and monomials are continuous functions). This shows that
\begin{align}
    f(z)\in \mathbb{C}[z] \implies f:\mathbb{C}^m\rightarrow \mathbb{C} \text{ is a continuous function.}
\end{align}

Now the key observations for the purposes of this paper, when $m\geq 2$, are the following: if $f(z)\in \mathbb{C}[z]$ and if we choose one variable, say, $z_j$ and fix all the values of rest of the variables, say, $z_i\in \mathbb{C}$ for all $i=1,\ldots, m$ with $i\not=j$ then the function $p$ defined in terms of $f(z)$ by $p(z_j)=f(z)$ is a single variable polynomial in the variable $z_j$, i.e.,
\begin{align}
    z_i\in \mathbb{C}, \forall i\not= j\implies p(z_j)=f(z)\in \mathbb{C}[z_j],
\end{align}
and hence can be written as (\ref{Def_Poly_p}), but in the variable $z_j$, as
\begin{align}
    p(z_j)= a_{n_j}(w)z_j^{n_j} +a_{n_j-1}(w)z_j^{n_j-1}  + \cdots + a_1(w)z_j+ a_0(w),
\end{align}
where 
\begin{align}
    w\in \mathbb{C}^{m-1},\;\;w=(w_1,\ldots, w_{m-1}),\;\;w_i=\left\{ 
\begin{array}{cc}
z_{i}, & \text{if }1\leq i<j, \\ 
z_{i+1}, & \text{if }j\leq i\leq m-1,
\end{array}
\right. 
\end{align}
and $n_j$ is a nonnegative integer that can be taken to be independent of $w$ and $z_j$ [but of course depends on $f$ and $j$, for instance, one can take $n_j=0$ if $f$ is the zero polynomial otherwise, you can take $n_j=\deg_{z_j}f$ as defined in terms of the representation (\ref{DefMultiPolyMonomLinearComb}) of $f$ by (\ref{DefDegreezjInf})].
In particular, if we now allow $w$ to vary then $a_{\ell}(w)$ is a $(m-1)$-variable polynomial in $w$ for each $\ell =0, \ldots , n_j$, i.e.,
\begin{align}
    a_{\ell}(w)\in \mathbb{C}[w],\;\; \forall \ell=0,\ldots, n_j.
\end{align}
As such, since functions of the $w$ are continuous, it follows that
\begin{align}
    a_{\ell}:\mathbb{C}^{m-1}\rightarrow \mathbb{C} \text{ is a continuous function, }\forall \ell=0,\ldots, n_j.
\end{align}

Below we provide an example of a couple multivariate polynomials using our notation.
\begin{example}\label{ExMultvarPolyNotation}
Consider the following examples of a two variable polynomial written in our notation as 
\begin{align*}
    f(z_1,z_2) = \sum_{\beta \in (\mathbb{N}\cup \{0\})^2} a_{\beta} z^{\beta}\in \mathbb{C}[z_1,z_2].
\end{align*}
First, using the example polynomial $f(z_1,z_2)=2z_2 + 4z_1 -3$, then
\begin{align*}
    2z_2 + 4z_1 -3 = a_{(0,0)}z^{(0,0)} + a_{(1,0)}z^{(1,0)} + a_{(0,1)}z^{(0,1)}, 
\end{align*} where $a_{(0,0)} = -3$,  $a_{(1,0)} = 4$, $a_{(0,1)} = 2$, and $a_{(\beta(1),\beta(2))} = 0$ for all other $\beta\in (\mathbb{N}\cup \{0\})^2$. Another example is $f(z_1,z_2)=6z_1^2 - 8z_1z_2$ which can be written in our notation as
\begin{align*}
     6z_1^2 - 8z_1z_2 = a_{(2,0)}z^{(2,0)} + a_{(1,1)}z^{(1,1)} 
\end{align*} where $a_{(2,0)} = 6$, $a_{(1,1)} = -8$, and $a_{(\beta(1),\beta(2))} = 0$ for all other $\beta\in (\mathbb{N}\cup \{0\})^2$.
\end{example}

\subsection{Stability of multivariate polynomials}
Let $\Omega\subseteq \mathbb{C}^m$ be an open set. An $m$-variable polynomial $f\in \mathbb{C}[z]$ [where $z=(z_1,\ldots, z_m)$] is called a $\Omega$\textit{-stable polynomial} \cite{09BBI, 09BBII, 11DW} if $f$ has no zeros in $\Omega$, i.e., $f(z)\not =0$ for all $z\in \Omega$ (in \cite[Sec. 9.1]{11DW} the zero polynomial is considered stable, which is atypical, and is not considered stable according to our definition from \cite{09BBI, 09BBII}). Typically, $\Omega$ is a polydomain, i.e., the Cartesian product of $m$ open connected sets in $\mathbb{C}$. In this case, a $\Omega$-stable polynomial is called, more precisely, a \textit{widest-sense $\Omega$-Hurwitz} polynomial in order to distinguish it from other more restrictive classes of $\Omega$-stable polynomials such as \textit{scattering $\Omega$-Hurwitz} or \textit{strictest-sense $\Omega$-Hurwitz} \cite[p. 619]{90SB}. 

Consider the following important examples. If $\mathbb{D}=\{z\in\mathbb{C}:|z|<1\}$ is the open unit disc in $\mathbb{C}$ then $\Omega= \mathbb{D}^m$ is called the open polydisc and a $\Omega$-stable polynomial is called \textit{Schur stable} \cite{17NB, 11DW} or, more precisely, a \textit{widest-sense Schur} polynomial \cite{90SB, 87FBII}. If $\mathbb{H}=\{z\in\mathbb{C}:\operatorname{Re}z>0\}$ is the open right-half plane in $\mathbb{C}$ then $\Omega= \mathbb{H}^m$ is called the open right polyhalfplane and a $\Omega$-stable polynomial is called \textit{Hurwitz stable} \cite{17NB, 07PB, 11DW} or, more precisely, a \textit{widest-sense Hurwitz} polynomial \cite{90SB,86AF, 87FBI}. In addition, they are referred to as \textit{weakly Hurwitz stable} polynomials \cite{09BBI, 09BBII} or as polynomials with the \textit{half-plane property} \cite{07PB, 04CO}. Other common examples for $\Omega$-stable polynomials come from rotation of $\mathbb{H}$ by an angle $\theta\in \mathbb{R}$ in the complex plane, i.e., $\mathbb{H}_{\theta}=e^{i\theta}\mathbb{H}=\{e^{i\theta}z:z\in \mathbb{H}\}$, and then $\Omega=\mathbb{H}_{\theta}^m$ is called an open polyhalfplane. For instance, when $\theta=\frac{\pi}{2}$ then $\mathbb{H}_{\frac{\pi}{2}}=\{z\in \mathbb{C}:\operatorname{Im}z>0\}$ is the open upper-half plane in $\mathbb{C}$ and $\Omega=\mathbb{H}_{\frac{\pi}{2}}^m$ is called the open upper polyhalfplane and $\Omega$-stable polynomials are just called \textit{stable} in \cite{11RP, 11DW}.

We now consider the most general open set $\Omega\subseteq \mathbb{C}^m$ which is a Cartesian product of open sets in $\mathbb{C}$. Let $B_j$, $1\leq j \leq m$ be open subsets of $\mathbb{C}$. Denote their boundaries by $\partial B_j$, their closure by $\overline{B_j}$, and their Cartesian product $\Omega=\prod_{j=1}^{m} B_j$ by $$\prod_{j=1}^{m} B_j = B_1\times \cdots \times B_m = \left\{(z_1, \ldots ,z_m) \in \mathbb{C}^m \; |\; z_j \in B_j \text{ for } j=1,\ldots , m \right\}.$$ This next theorem from \cite[Lemma 5.2]{11RP} (see also the appendices in \cite{86AF, 84FL}) is a result in multivariate stability theory which describes an important consequence for such $\Omega$-stable polynomials (see, for instance, \cite{86AF}) that has direct applications in the study of hyperbolic polynomials (see, for instance, \cite{11RP}) and passive electrical network theory on realizability and synthesis problems (see, for instance, \cite{77NB, 84FL, 91AK, 68TK, 15DY}).
\begin{theorem}[On the boundary zeros of $\boldsymbol{\Omega}$-stable polynomials]\label{11RPThm}
Suppose a polynomial $f \in \mathbb{C}[z]$, where $z=(z_1,\ldots, z_m)$, such that $f(z)\not=0$ for all $z\in \prod_{j=1}^mB_j$. Then either 
\begin{align} \label{Hypothesis1}
    f(z)&\not=0, \text{ for all }z \in \prod_{j=1}^m\overline{B_j}
\end{align} 
\begin{center}
    or
\end{center}
\begin{align}\label{Hypothesis2}
    f(\alpha)=0, \text{ for some }\alpha\in  \prod_{j=1}^m \overline{B_j}
\end{align} 
in which case one of the following two cases occur: 
\begin{itemize}
   \item[1.]  $\alpha \in \prod_{j=1}^m\partial B_j$.
    \item[2.]  There exists a nonempty set $S\subseteq \{1,\ldots, m\}$ such that 
    \begin{align*}
    \alpha_j\in B_j, \text{ for all } j\in S
    \end{align*} 
    \begin{center}
        and
    \end{center}
  \begin{align*}
  f(z)=0 \text{ for all } z\in\mathbb{C}^n \text{ with } z_j=\alpha_j \text{ for all }j\not\in S.
  \end{align*} [i.e., the function $f(z)$, with the coordinates $z_j=\alpha_j$ fixed for all $j\not\in S$, is identically equal to $0$ as a function of the remaining $z_j$-variables with $j\in S$.]
\end{itemize} 
\end{theorem}
\begin{proof} The proof will be by induction on $m$. We begin by treating the base case $m=1$. Suppose that $f \in \mathbb{C}[z]$, where $z=z_1$ such that $f(z)\not=0$ for all $z\in B_1$. Then obviously either $f(z)\not =0$ for all $z\in \overline{B_1}$ or $f(\alpha)=0$ for some $\alpha\in B_1$. In the latter case, since $\overline{B_1}=B_1\cup \partial B_1$ and by the hypothesis $f(z)\not=0$ for all $z\in B_1$, it follows that $\alpha\in \partial B_1$. Thus, the statement is true for $m=1$. Suppose that the statement is true for some natural number $m$. We will now prove the statement for $m+1$. Hence, let $f \in \mathbb{C}[z]$, where $z=(z_1,\ldots,z_{m+1})$, such that $f(z)\not=0$ for all $z\in \prod_{j=1}^{m+1}B_j$. If \eqref{Hypothesis1} is true we are done. Thus, suppose \eqref{Hypothesis2} is true. Hence, let  $ \alpha\in  \prod_{j=1}^{m+1} \overline{B_j}$ with $f(\alpha)=0$. If $\alpha \in \prod_{j=1}^{m+1}\partial B_j$ then we are done. Hence, suppose $\alpha \not\in \prod_{j=1}^{m+1}\partial B_j$. This implies for some $i\in\{1, \ldots , m+1\}$ that $\alpha_i\not\in \partial B_i$, but since $\alpha_i\in \overline{B_i}= B_i \cup \partial B_i$, it follows that $\alpha_i\in B_i$. Let $S$ be the set of all $j\in \{1, \ldots,m+1\}$ such that $\alpha_j\in B_j$. Hence, if $j\not \in S$ then $\alpha_j\in \partial B_j$. In particular, it follows that $S\subseteq \{1, \ldots,m+1\}$ and $\emptyset \not = S\not = \{1,\ldots , m+1\}$. It remains to prove that if $z\in\mathbb{C}^{m+1}$ and $z_j=\alpha_j$ for all $j\not\in S$ then $f(z)=0$. Choose an arbitrary $i\in S$. Without loss of generality (by permutating the variables) we can assume that $i=1$ and $S=\{1,\ldots, k\}$. Denote $\boldsymbol{z_2}=(z_2,\ldots ,z_{m+1})$, $\boldsymbol{\alpha_2} = (\alpha_2, \ldots , \alpha_{m+1})$, and $f(z_1,\boldsymbol{z_2})=f(z_1,z_2,\ldots, z_{m+1})$. To apply Theorem \ref{Thm1}, we choose the polynomials in the single variable $z_1$ to be 
\begin{align*}
    q=q(z_1)=f(z_1,\boldsymbol{\alpha_2}),\;\;p=p(z_1)=f(z_1,\boldsymbol{z_2}),
\end{align*}
for fixed $\boldsymbol{z_2}\in \mathbb{C}^m$. More precisely, $q$ and $p$ can be expressed as
\begin{align*}
    q(z_1)= f(z_1,\boldsymbol{\alpha_2}) &= a_{n}(\boldsymbol{\alpha_2})z_1^{n} +a_{n-1}(\boldsymbol{\alpha_2})z_1^{n-1}  + \cdots + a_1(\boldsymbol{\alpha_2})z_1+ a_0(\boldsymbol{\alpha_2}), \\
    p(z_1)= f(z_1,\boldsymbol{z_2}) &= a_{n}(\boldsymbol{z_2})z_1^{n} +a_{n-1}(\boldsymbol{z_2})z_1^{n-1}  + \cdots + a_1(\boldsymbol{z_2})z_1+ a_0(\boldsymbol{z_2}),
\end{align*}
where each $\boldsymbol{z_2} \mapsto a_{\ell}(\boldsymbol{z_2})$ is a polynomial in $\boldsymbol{z_2}$ for each $\ell =0, \ldots , n$ and as such they are all continuous at $\boldsymbol{z_2}=\boldsymbol{\alpha_2}$. 
We now want to show $q(z_1)$ is the zero polynomial, i.e., $q(z_1)=0$, using $p(z_1)$ as an approximation to it when $\boldsymbol{z_2}$ is sufficiently near $\boldsymbol{\alpha_2}$. To do this we give a proof by contradiction. Assume that $q(z_1)$ is not the zero polynomial. Our goal is to use Theorem \ref{Thm1} with this polynomial $q$ and the root $\zeta=\alpha_1$ to derive the contradiction. First, as $\alpha_1\in B_1$ and $B_1$ is open in $\mathbb{C}$, then there exists an $\varepsilon>0$ such that if $|z_1-\zeta|<\varepsilon$ then $z_1\in B_1$. Second, by Theorem \ref{Thm1} there exists a $\delta>0$ such that if $\max_{0\leq \ell \leq n}|a_{\ell}(\boldsymbol{\alpha_2})-a_{\ell}(\boldsymbol{z_2})|<\delta$ then $p(z_1)$ has at least one root [denote it by $\lambda=\lambda(\boldsymbol{z_2})$] whose distance from $\zeta$ is less than $\varepsilon$, i.e., $|\lambda - \zeta |<\varepsilon.$ In particular, by continuity of the polynomials $a_{\ell}(\boldsymbol{z_2})$, $\ell=0,\ldots, n$ at $\boldsymbol{z_2}=\boldsymbol{\alpha_2}$, there exists a $\delta^{\prime}$ such that if $|\boldsymbol{z_2}-\boldsymbol{\alpha_2}|<\delta^{\prime}$ then $\max_{0\leq \ell \leq n}|a_{\ell}(\boldsymbol{\alpha_2})-a_{\ell}(\boldsymbol{z_2})|<\delta$ and hence the root $\lambda=\lambda(\boldsymbol{z_2})$ of $p(z_1)$ satisfies $|\lambda - \zeta |<\varepsilon$ and so $\lambda \in B_1$. But then, since $\boldsymbol{\alpha_2}\in\prod _{j=2}^{m+1}\overline{B_j}$, we choose a $\boldsymbol{z_2}\in\prod _{j=2}^{m+1}B_j$ with $|\boldsymbol{z_2}-\boldsymbol{\alpha_2}|<\delta^{\prime}$ such that $0=p(\lambda)=f(\lambda,\boldsymbol{z_2})$ and $(\lambda,\boldsymbol{z_2})\in \prod _{j=1}^{m+1}B_j$, a contradiction of the hypothesis that $f$ is never zero in $\prod _{j=1}^{m+1}B_j$. This proves $q(z_1)=f(z_1,\boldsymbol{\alpha_2})$ is the zero polynomial, as desired.
Consider the polynomial $g(\boldsymbol{z_2})= f(z_1, \boldsymbol{z_2}) \in \mathbb{C}[\boldsymbol{z_2}]$ for a fixed $z_1\in B_1$. It depends on the $m$-variables $\boldsymbol{z_2}=(z_2,\ldots, z_{m+1})$, $g(\boldsymbol{z_2})\not=0$ for all $\boldsymbol{z_2}\in \prod_{j=2}^{m+1}B_j$, and $g(\boldsymbol{\alpha_2})=0$ with $\boldsymbol{\alpha_2}\in \prod _{j=2}^{m+1}\overline{B_j}$, $\alpha_j\in B_j$ for all $j\in S$, i.e., for all $j=1,\ldots, k$, and $\alpha_j\in \partial B_j$ for all $j\not \in S$, i.e., for all $j=k+1, \ldots, m+1$. Applying the induction hypothesis to the polynomial $g$ and its roots $\boldsymbol{\alpha_2}$ there are two possible cases. The first case is $\boldsymbol{\alpha_2}\in\prod _{j=2}^{m+1}\partial B_j$. Then as we showed $q(z_1)=f(z_1, \boldsymbol{\alpha_2})=g(\boldsymbol{\alpha_2})=0$ is the zero polynomial, i.e., $f(z_1, \boldsymbol{\alpha_2})=0$ for all $z\in \mathbb{C}$, and hence the statement is true in this case. The second possible case is $\boldsymbol{\alpha_2}\not\in\prod _{j=2}^{m+1}\partial B_j$, in which case $g(z_2,\ldots, z_k, \alpha_{k+1},\ldots, \alpha_{m+1})=0$ for all $z_2,\ldots, z_k\in \mathbb{C}$. Thus, the second case must be true for every $z_1\in B_1$ implying that $p(z_1)=f(z_1, z_2, \ldots, z_k, \alpha_{k+1}, \ldots , \alpha_{m+1})=0$ for all $z_2, \ldots, z_{k} \in \mathbb{C}$ and all $z_1\in B_1$. As $p(z_1)$ is a polynomial in $z_1$ for each fixed $z_2, \ldots, z_{k} \in \mathbb{C}$ which has an infinite number of roots (for instance, all those in $B_1$) then it must be the zero polynomial so that $f(z_1, z_2, \ldots, z_k, \alpha_{k+1}, \ldots, \alpha_{m+1})=0$ for all $z_1, z_2, \ldots, z_{k} \in \mathbb{C}$. This proves the statement in this case. As we have exhausted all cases, the statement is true for $m+1$. Therefore, by induction the statement is true for all $m$.  \end{proof}

\begin{example}[Illustration of Case (\ref{Hypothesis2}).2 in Theorem \ref{11RPThm} when $\boldsymbol{m=2}$]  \label{Ex4m=2}
Consider a polynomial $f\in \mathbb{C}[z]$, where $z=(z_1,z_2)$, such that $f(z)\not=0$ for all $z\in B_1 \times B_2$ and $f(\alpha)=0$ for some $\alpha=(\alpha_1,\alpha_2) \in \overline{B_1}\times \overline{B_2}$. Suppose, as in Theorem \ref{11RPThm} Case (\ref{Hypothesis2}).2, that $S=\{1\}$, in which case $\alpha_1\in B_1$ and $\alpha_2\in \partial B_2$. The goal of Theorem \ref{11RPThm} when $m=2$ is to answer the question: Does $f(z_1,\alpha_2)=0$ for all $z_1\in \mathbb{C}$? This question is shown in Figure \ref{firstfigz1z2plane} by the question mark to the right of the red dashed line that represents $f(z_1, \alpha_2)$, for all $z_1\in\mathbb{C}$.
\begin{figure}[!t]
   \centering
    \begin{tikzpicture}[scale=1]
        \draw[>=stealth, ->] (0,-0.5) --(0,5) node[above] {$z_2$};
        \draw[>=stealth, ->] (-0.5,0) --(5,0) node[right] {$z_1$};
          \draw[dashed, thick] (1,1) -- (1,4);
           \draw[dashed, thick, color=orange] (3.7,0) -- (3.7,4);
          \draw[dashed, thick] (4,4) -- (4,1);
          \draw[dashed, thick, color=blue] (4,1) -- (1,1);
        \draw[bicolor={blue}{red}] (1,4) -- (4,4);
        \draw[dashed, thick, color=red, <-, >=stealth] (-1.5,4) -- (1,4);
        \draw[dashed, thick, color=red, ->, >=stealth] (4,4) -- (6,4) node[color=black, right, yshift=0.1cm] {\small{$f(z_1, \alpha_2)\overset{?}{=}0$}};
         \draw (1,0) node {$($} (4,0) node {$)$} (2.5,0) node[below, yshift=-0.3cm] {$B_1$};
        \draw (0,1) node[rotate=90] {$($} (0,4) node[rotate=90] {$)$} (0, 2.5) node[left, xshift=-0.3cm] {$B_2$};
        \draw[dashed, color=green, thick] (1.5,0) -- (1.5,2);
        \draw[dashed, color=green, thick] (0,2) -- (1.5,2);
        \draw [color=black, fill=black] (0,4.01)circle(0.05);
        \draw [color=black] (0,4) node[left, xshift=-0.2cm] {\small{$\alpha_2$}};
        \draw [color=black, fill=black] (1.5,0)circle(0.05)  node[below] {\small{$z_1$}};
        \draw [color=black, fill=black] (0,2)circle(0.05)  node[left] {\small{$z_2$}};
        \draw [color=black, fill=black] (3.7,0)circle(0.05);
        \draw [color=black] (3.7,0) node[below] {\small{$\alpha_1$}};
        \draw [color=black, fill=black] (3.7,4)circle(0.05) node[above, xshift=0.68cm] {\small{$\alpha$, $f(\alpha)=0$}};
        \draw [color=black, fill=black] (1.5,2)circle(0.05) node[right] {\small{$z,f(z)\!\not=\!0$}};
    \end{tikzpicture}  
    \caption{In this figure we illustrate $\mathbb{C}^2$ as just a $z_1z_2$-plane when $S=\{1\}$. We can visualize the polynomial $f(z)\not=0$, for any $z_1,z_2\in B_1 \times B_2$ at the point $z$ by the green dashed lines and the zero polynomial $f(\alpha)=0$ when $\alpha_1 \in B_1$ and $\alpha_2 \in \partial B_2$ at the point $\alpha$ where the orange dashed line intersects the red dashed line. The boundaries of $B_1$, $\partial B_1$, are the black dashed lines and the boundaries of $B_2$, $\partial B_2$, are denoted by the blue dashed lines. Finally, the polynomial $f(z_1, \alpha_2)$, for a fixed $z_2=\alpha_2$ when $\alpha_2\in \partial B_2$, is the single variable polynomial in $z_1$ denoted by the red dashed line.}
    \label{firstfigz1z2plane}
\end{figure}
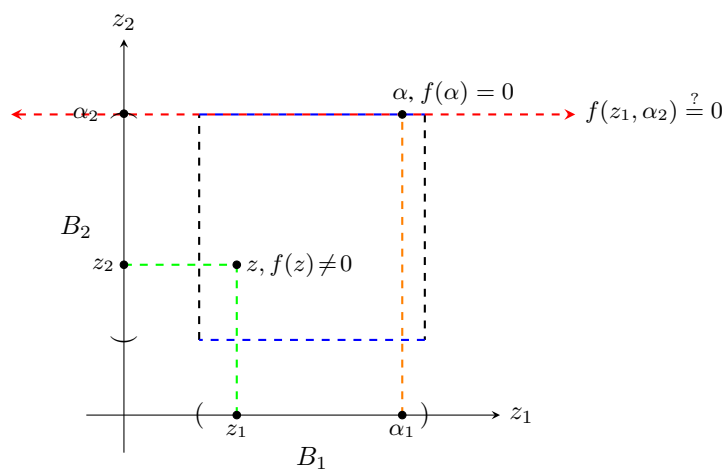
Following the proof of Theorem \ref{11RPThm}, for a fixed $z_2$, we have $z_1 \mapsto f(z_1,z_2)$ is a single variable polynomial in the variable $z_1$. To apply Theorem \ref{Thm1}, we choose the polynomials in the single variable $z_1$ to be 
\begin{align*}
    q=q(z_1)=f(z_1,\alpha_2),\;\;p=p(z_1)=f(z_1,z_2),
\end{align*}
for fixed $z_2\in \mathbb{C}$ so that it is implicitly parameterized by $z_2$. More precisely, $q$ and $p$ can be expressed in the following forms:
\begin{align*}
    q(z_1)= f(z_1,\alpha_2) &= a_n(\alpha_2)z_1^n +a_{n-1}(\alpha_2)z_1^{n-1}  + \cdots + a_1(\alpha_2)z_1+ a_0(\alpha_2), \\
    p(z_1)= f(z_1,z_2) &= a_n(z_2)z_1^n +a_{n-1}(z_2)z_1^{n-1}  + \cdots + a_1(z_2)z_1+ a_0(z_2).
\end{align*}
Note that each $z_2 \mapsto a_{\ell}(z_2)$ is a polynomial in $z_2$ for each $\ell=0, \ldots , n$ and as such they are all continuous at $z_2=\alpha_2$. 

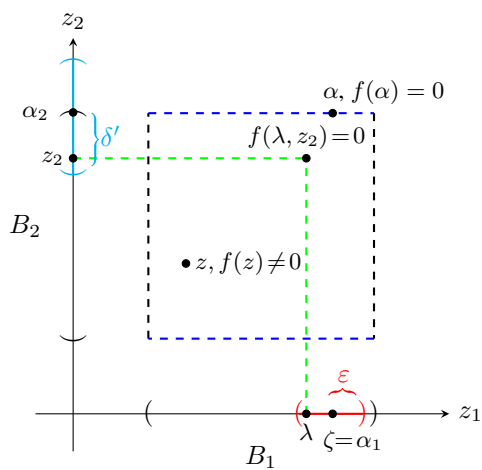
\begin{figure}[h!t]
   \centering
    \begin{tikzpicture}[scale=1]
        \draw[>=stealth, ->] (0,-0.5) --(0,5) node[above] {$z_2$};
        \draw[>=stealth, ->] (-0.5,0) --(5,0) node[right] {$z_1$};
          \draw[dashed, thick] (1,1) -- (1,4);
          \draw[dashed, thick] (4,4) -- (4,1);
           \draw [color=cyan, thick] (0,3.2) -- (0,4.7);
      \draw [color=red, thick] (3,0) -- (3.85,0);
          \draw[dashed, thick, color=blue] (4,1) -- (1,1);
        \draw[color=blue, dashed, thick] (1,4) -- (4,4);
         \draw (1,0) node {$($} (4,0) node {$)$} (2.5,0) node[below, yshift=-0.3cm] {$B_1$};
          \draw[color=red] (3,0) node {$($} (3.85,0) node {$)$};
        \draw (0,1) node[rotate=90] {$($} (0,4) node[rotate=90] {$)$} (0, 2.5) node[left, xshift=-0.3cm] {$B_2$};
        \draw[dashed, color=green, thick] (3.1,0) -- (3.1,3.4);
        \draw[dashed, color=green, thick] (0,3.4) -- (3.1,3.4);
        \draw [color=black, fill=black] (0,4.01)circle(0.05);
        \draw [color=black] (0,4) node[left, xshift=-0.2cm] {\small{$\alpha_2$}};
        \draw [color=black, fill=black] (3.1,0)circle(0.05)  node[below] {\small{$\lambda$}};
        \draw [color=black, fill=black] (0,3.4)circle(0.05)  node[left] {\small{$z_2$}};
        \draw [color=black, fill=black] (3.45,0)circle(0.05);
        \draw [color=black] (3.72,0) node[below, yshift=-0.1cm] {\small{$\zeta \! \!=\!\alpha_1$}};
        \draw [color=black, fill=black] (3.45,4)circle(0.05) node[above, xshift=0.68cm] {\small{$\alpha$, $f(\alpha)=0$}};
        \draw [color=black, fill=black] (1.5,2)circle(0.05) node[right] {\small{$z,f(z)\!\not=\!0$}};
        \draw [pen colour={red}, thick, decorate,
    decoration = {calligraphic brace}] (3.4,0.25) --  (3.76,0.25);
      \draw [color=black, fill=black] (3.1,3.4)circle(0.05) node[above] {\small{$f(\lambda , z_2)\!=\!0$}};
    \draw[color=red] (3.6,0.5) node {$\varepsilon$};
    \draw [pen colour={cyan}, thick, decorate,
    decoration = {calligraphic brace, mirror}] (0.25,3.3) --  (0.25,4);
     \draw[color=cyan] (0.5,3.7) node {$\delta '$};
      \draw[color=cyan] (0,3.2) node[rotate=90] {$($} (0,4.7) node[rotate=90] {$)$};
    \end{tikzpicture}  
    \caption{In this figure we illustrate both the $\delta'$-ball $B(\alpha_2,\delta')$ (cyan) and the $\varepsilon$-ball $B(\zeta , \varepsilon)$ (red) on the $z_1z_2$-plane when $S=\{1\}$. Theorem \ref{Thm1} guarantees, for some $z_2\in B_2\cap B(\alpha_2,\delta')$, at least one $\lambda=\lambda(z_2)$ within the $\varepsilon$-ball $B(\zeta , \varepsilon)\subseteq B_1$ such that $f(\lambda, z_2)=0$. A contradiction that $f(z)\not=0$ for all $z\in B_1 \times B_2$.}
    \label{firstfigz1z2plane2}
\end{figure}

We now want to show that $q(z_1)$ is the zero polynomial, i.e., $q(z_1)=0$, using $p(z_1)$ as an approximation to it when $z_2$ is sufficiently near $\alpha_2$. To do this we give a proof by contradiction. Assume that $q(z_1)$ is not the zero polynomial, i.e., $q(z_1)\not=0$. We will now use Theorem \ref{Thm1} with this polynomial $q$ and the root $\zeta=\alpha_1$. Next, as $\alpha_1\in B_1$ and $B_1$ is open in $\mathbb{C}$, then there exists an $\varepsilon>0$ such that if $|z_1-\zeta|<\varepsilon$ then $z_1\in B_1$. Now by Theorem 1, there exists a $\delta>0$ such that if $\max_{0\leq \ell \leq n}|a_{\ell}(\alpha_2)-a_{\ell}(z_2)|<\delta$ then $p(z_1)$ has at least one root [denote it by $\lambda=\lambda(z_2)$] whose distance from $\zeta$ is less than $\varepsilon$, i.e., $|\lambda - \zeta |<\varepsilon.$ In particular, by continuity of the polynomials $a_{\ell}(z_2)$, $\ell =0,\ldots, n$ at $z_2=\alpha_2$, there exists a $\delta^{\prime}$ such that if $|z_2-\alpha_2|<\delta^{\prime}$ then $\max_{0\leq \ell \leq n}|a_{\ell}(\alpha_2)-a_{\ell}(z_2)|<\delta$ and hence the root $\lambda=\lambda(z_2)$ of $p(z_1)$ satisfies $|\lambda - \zeta |<\varepsilon$ and so $\lambda \in B_1$. But then, since $\alpha_2\in\partial B_2$, we choose a $z_2\in B_2$ with $|z_2-\alpha_2|<\delta^{\prime}$ such that $0=p(\lambda)=f(\lambda,z_2)$ and $(\lambda,z_2)\in B_1\times B_2$ (this is illustrated in Figure \ref{firstfigz1z2plane2} below), a contradiction of the hypothesis that $f$ is never zero in $B_1\times B_2$. This proves $q(z_1)=f(z_1,\alpha_2)$ is the zero polynomial, as desired. This completes the example of Case (\ref{Hypothesis2}).2 in Theorem \ref{11RPThm} when $m=2$.
\end{example}


\begin{thebibliography}{1}

\bibitem{90SB}Basu, S. (1990). On boundary implications of stability and positivity properties of multidimensional systems. \textit{Proceedings of the IEEE.} 78(4): 614-626. DOI: 10.1109/5.54802

\bibitem{87FBII}Basu, S., Fettweis, A. (1987). New results on stable multidimensional polynomials- Part II: Discrete case. \textit{IEEE Trans. Circuits Syst.}, 34(11): 1264-1274. DOI: 10.1109/TCS.1987.1086065

\bibitem{99BO}Bender, C., Orszag, S. (1999). \textit{Advanced Mathematical Methods for Scientistics and Engineers I: Asymptotic Methods and Perturbation Theory.} New York: Springer-Verlag. DOI: 10.1007/978-1-4757-3069-2

\bibitem{09BBI}Borcea, J., Brändén, P. (2009). The Lee-Yang and Pólya-Schur programs. I. Linear operators preserving stability. \textit{Invent. math.} 177: 541-561. DOI: 10.1007/s00222-009-0189-3

\bibitem{09BBII}Borcea, J., Brändén, P. (2009). The Lee-Yang and Pólya-Schur programs. II. Theory of stable polynomials and applications. \textit{Comm. Pure Appl. Math.} 62: 1595-1631. DOI: 10.1002/cpa.20295

\bibitem{77NB}Bose, N. (1977). Problems and Progress in Multidimensional Systems Theory. \textit{Proc. IEEE.} 65(6): 824-840. DOI: 10.1109/PROC.1977.10579

\bibitem{17NB}Bose, N. (2017). \textit{ Applied Multidimensional Systems Theory}, 2nd Ed., Springer. DOI: 10.1007/978-3-319-46825-9

\bibitem{07PB} Br\"andén, P. (2007). Polynomials with the half-plane property and matroid theory. \textit{Advances in Mathematics}. 216(1): 302-320. DOI: 10.1016/j.aim.2007.05.011

\bibitem{04CO}Choe, Y., Oxley, J., Sokal, A., Wagner, D. (2004). Homogeneous multivariate polynomials with the half-plane property. \textit{Advances in Applied Mathematics}. 32(1): 88-187. DOI: 10.1016/S0196-8858(03)00078-2

\bibitem{89CC}Cucker, F., Corbalan, A. G. (1989). An Alternate Proof of the Continuity of the Roots of a Polynomial. \textit{The American Mathematical Monthly}. 96(4): 342-345. DOI:
10.1080/00029890.1989.11972193



\bibitem{86AF}Fettweis, A. (1986). A new approach to Hurwitz polynomials in several variables. \textit{Circuits Systems and Signal Process.} 5(4): 405–417. DOI: 10.1007/BF01599617

\bibitem{84FL}Fettweis, A., Linnenberg, G. (1984). An extension of the maximum-modulus principle for applications to multidimensional networks. \textit{AE\"{U} - International Journal of Electronics and Communications.} 38(2): 131-135. 

\bibitem{87FBI}Fettweis, A., Basu, S. (1987). New results on multidimensional stable polynomials - Part I: Continuous case. \textit{IEEE Trans. Circuits Syst.} 34(10): 1221–1232. DOI: 10.1109/TCS.1987.1086057


\bibitem{20KH}Hirose, K., (2020). Continuity of the roots of a polynomial. \textit{The American Mathematical Monthly} 127(4): 359-363. DOI: 10.1080/00029890.2020.1704166

\bibitem{97HM}Hirst, H. P., Macey, W. T. (1997). Bounding the roots of polynomials. \textit{The College Math. Journal}, 28(4): 292-295. DOI: 10.1080/07468342.1997.11973878

\bibitem{68TK}Koga, T. (1968). Synthesis of finite passive n-ports with prescribed positive real matrices of several variables. \textit{IEEE Trans. Circuit Theory.} 15(1): 2-23. DOI: 10.1109/TCT.1968.1082780

\bibitem{91AK}Kummert, A. (1991). On the synthesis of multidimensional reactance multiports. \textit{IEEE Transactions on Circuits and Systems.} 38(6): 637–642. DOI: 10.1109/31.81858 

\bibitem{66MM}Marden, M. (1966). \textit{Geometry of Polynomials.} 2nd Ed., Providence: American Mathematical Society.


\bibitem{11RP}Pemantle, R. (2012). Hyperbolicity and stable polynomials in combinatorics and probability. In: Jerison, Mazur, Mrowka, Schmid and Stanley, ed(s). \textit{Current Developments in Mathematics, Proceedings of the 2011 conference}. International Press: Somerville, MA., pages 57-124. DOI: 10.4310/CDM.2011.v2011.n1.a2

\bibitem{76WR}Rudin, W. (1976). \textit{Principles of Mathematical Analysis.} 3rd Ed., McGraw-Hill, Inc.

\bibitem{90SS}Stewart, G. W., Sun, J. (1990). \textit{Matrix perturbation theory.} Boston: Academic Press.

\bibitem{77US}Uherka, D. J., Sergott, A. M. (1977). On the continuous dependence of the roots of a polynomial on its coefficients. \textit{The American Mathematical Monthly}, 84(5):368-370. DOI: 10.1080/00029890.1977.11994362

\bibitem{11DW}Wagner, D. (2011). Multivariate stable polynomials: theory and applications. \textit{Bull. Amer. Math. Soc.} 48: 53-84. DOI: 10.1090/S0273-0979-2010-01321-5

\bibitem{15DY}Youla, D. (2015). 
\textit{Theory and Synthesis of Linear Passive Time-Invariant Networks.} Cambridge: Cambridge University Press. DOI: 10.1017/CBO9781316403105

\bibitem{65MZ}Zedek, M. (1965). Continuity and location of zeros of linear combinations of polynomials. \textit{Proc. Amer. Math. Soc.} 16: 78-84. DOI: 10.1090/S0002-9939-1965-0171902-8

\end{thebibliography}
\end{document}